\documentclass[a4paper,10pt]{article}
\usepackage[latin1]{inputenc}
\usepackage[T1]{fontenc}
\usepackage{lmodern}
\usepackage{a4wide}
\usepackage{empheq}
\usepackage{amssymb,amsthm}
\usepackage{enumerate,bbm}
\usepackage[a4paper]{hyperref}

\newcommand{\m}{\mbox}
\newcommand{\eps}{\varepsilon}
\newcommand{\px}{\partial_x}
\newcommand{\phib}{\varphi}
\renewcommand{\leq}{\leqslant}
\renewcommand{\geq}{\geqslant}

\newcommand{\vers}{\longrightarrow}

\newcommand{\cvf}{\rightharpoonup}

\newcommand{\lent}{[\kern-0.15em[}
\newcommand{\rent}{]\kern-0.15em]}
\newcommand{\ient}[2]{\lent #1,#2\rent}
\newcommand{\unn}{\ient{1}{N}}

\newcommand{\nh}[1]{{\left\|#1\right\|}_{H^1}}
\newcommand{\nlu}[1]{{\left\|#1\right\|}_{L^1}}
\newcommand{\nld}[1]{{\left\|#1\right\|}_{L^2}}
\newcommand{\nli}[1]{{\left\|#1\right\|}_{L^{\infty}}}
\newcommand{\nhs}[1]{{\left\|#1\right\|}_{H^s}}

\newcommand{\carre}[1]{{#1}^2}
\newcommand{\Carre}[1]{{\left( #1 \right)}^2}
\newcommand{\puiss}[2]{{#1}^{#2}}
\newcommand{\Puiss}[2]{{\left( #1 \right)}^{#2}}

\newcommand{\R}{\mathbb{R}}
\newcommand{\M}{\mathcal{M}}
\newcommand{\N}{\mathcal{N}}
\renewcommand{\L}{\mathcal{L}}
\renewcommand{\S}{\mathbb{S}}

\newcommand{\nht}[1]{{\left\|#1\right\|}_{H^{\frac{3}{4}}}}
\newcommand{\htq}{H^{\frac{3}{4}}}
\newcommand{\supt}{\sup_{t\in [0,T]}}

\renewcommand{\a}{\mathfrak{a}^-}
\renewcommand{\b}{\mathfrak{b}}
\renewcommand{\c}{\gamma}
\newcommand{\alpham}{\boldsymbol{\alpha}^-}
\newcommand{\alphav}{\boldsymbol{\alpha}}
\newcommand{\ta}{T(\a)}
\newcommand{\an}{A_1,\ldots,A_N}
\newcommand{\aj}{A_1,\ldots,A_j}
\newcommand{\dt}{\frac{d}{dt}}

\newcommand{\wt}[1]{\widetilde{#1}}
\newcommand{\zt}{\wt z}
\newcommand{\Rt}{\wt R}
\newcommand{\Zt}{\wt Z}
\newcommand{\epst}{\wt\eps}
\newcommand{\alphat}{\wt{\alpha}}
\newcommand{\alphatv}{\wt{\boldsymbol{\alpha}}}

\theoremstyle{plain}
\newtheorem{theo}{Theorem}[section]
\newtheorem{prop}[theo]{Proposition}
\newtheorem{lem}[theo]{Lemma}
\newtheorem{cor}[theo]{Corollary}
\newtheorem{claim}[theo]{Claim}
\theoremstyle{definition}
\newtheorem{defi}[theo]{Definition}
\newtheorem{rem}[theo]{Remark}
\newtheorem{notation}[theo]{Notation}

\numberwithin{equation}{section}

\title{Multi-soliton solutions for the supercritical gKdV equations}
\author{Vianney Combet}
\date{Universit\'e de Versailles Saint-Quentin-en-Yvelines,
 Math\'ematiques, UMR 8100, \\
 45, av. des \'Etats-Unis,
 78035 Versailles Cedex, France\\
 vianney.combet@math.uvsq.fr}

\begin{document}

\maketitle

\begin{abstract}
For the $L^2$ subcritical and critical \eqref{eq:gKdV} equations, Martel \cite{martel:Nsoliton} proved the existence and uniqueness of multi-solitons. Recall that for any $N$ given solitons, we call multi-soliton a solution of \eqref{eq:gKdV} which behaves as the sum of these $N$ solitons asymptotically as $t\to+\infty$. More recently, for the $L^2$ supercritical case, Côte, Martel and Merle \cite{martel:Nsolitons} proved the existence of at least one multi-soliton. In the present paper, as suggested by a previous work concerning the one soliton case \cite{combet}, we first construct an $N$-parameter family of multi-solitons for the supercritical \eqref{eq:gKdV} equation, for $N$ arbitrarily given solitons, and then prove that any multi-soliton belongs to this family. In other words, we obtain a complete classification of multi-solitons for \eqref{eq:gKdV}.
\end{abstract}

\section{Introduction}

\subsection{The generalized Korteweg-de Vries equation} \label{subsec:gKdV}

We consider the generalized Korteweg-de Vries equation: \begin{equation} \label{eq:gKdV} \tag{gKdV} \begin{cases} \partial_t u+\px^3 u +\px(u^p)=0\\ u(0)=u_0\in H^1(\R) \end{cases} \end{equation} where $(t,x)\in\R^2$ and $p\geq 2$ is integer. The following quantities are formally conserved for solutions of \eqref{eq:gKdV}: \begin{gather*} \int u^2(t) = \int u^2(0)\quad \m{(mass)},\\ E(u(t)) = \frac{1}{2}\int u_x^2(t) -\frac{1}{p+1}\int u^{p+1}(t) = E(u(0))\quad \m{(energy)}. \end{gather*}

Kenig, Ponce and Vega \cite{kpv} have shown that the local Cauchy problem for \eqref{eq:gKdV} is well posed in $H^1(\R)$: for $u_0\in H^1(\R)$, there exist $T>0$ and a solution $u\in C([0,T],H^1(\R))$ of \eqref{eq:gKdV} satisfying $u(0)=u_0$ which is unique in some class $Y_T\subset C([0,T],H^1(\R))$. Moreover, if $T^*\geq T$ is the maximal time of existence of $u$, then either $T^*=+\infty$ which means that $u(t)$ is a global solution, or $T^*<+\infty$ and then $\nh{u(t)} \to +\infty$ as $t\uparrow T^*$ ($u(t)$ is a finite time blow up solution). Throughout this paper, when referring to an $H^1$ solution of \eqref{eq:gKdV}, we mean a solution in the above sense. Finally, if $u_0\in H^s(\R)$ for some $s\geq 1$, then $u(t)\in H^s(\R)$ for all $t\in [0,T^*)$.

In the case where $2\leq p<5$, it is standard that all solutions in $H^1$ are global and uniformly bounded by the energy and mass conservations and the Gagliardo-Nirenberg inequality. In the case $p=5$, the existence of finite time blow up solutions was proved by Merle \cite{merle} and Martel and Merle \cite{martel:blowup}. Therefore $p=5$ is the critical exponent for the long time behavior of solutions of \eqref{eq:gKdV}. For $p>5$, the existence of blow up solutions is an open problem.

We recall that a fundamental property of \eqref{eq:gKdV} equations is the existence of a family of explicit traveling wave solutions. Let $Q$ be the only solution (up to translations) of \[ Q>0,\quad Q\in H^1(\R),\quad Q''+Q^p=Q,\quad \m{i.e.}\ Q(x)=\Puiss{\frac{p+1}{2\cosh^2\left(\frac{p-1}{2}x\right)}}{\frac{1}{p-1}}. \] For all $c_0>0$ and $x_0\in\R$, \[ R_{c_0,x_0}(t,x)=Q_{c_0}(x-c_0t-x_0) \] is a solution of \eqref{eq:gKdV}, where $Q_{c_0}(x)=c_0^{\frac{1}{p-1}}Q(\sqrt{c_0}x)$. We call solitons these solutions though they are known to be solitons only for $p=2,3$ (in the sense that they are stable by interaction).

It is well-known that the stability properties of a soliton solution depend on the sign of
${\frac{d}{dc} \int Q_c^2}_{|c=c_0}$. Since $\int Q_c^2 = c^{\frac{5-p}{2(p-1)}} \int Q^2$, we distinguish the following three cases:
\begin{itemize}
	\item For $p<5$ ($L^2$ subcritical case), solitons are stable and asymptotically stable in $H^1$ in some suitable sense: see Cazenave and Lions \cite{cazenavelions}, Weinstein \cite{weinstein:lyapunov}, Grillakis, Shatah and Strauss \cite{gss} for orbital stability; and Pego and Weinstein \cite{pegoweinsteinbis}, Martel and Merle \cite{martel:revisited} for asymptotic stability.
	\item For $p=5$ ($L^2$ critical case), solitons are unstable, and blow up occurs for a large class of solutions initially arbitrarily close to a soliton, see \cite{martel:blowup,merle}. Moreover, for both critical and subcritical cases, previous works imply the following asymptotic classification result: if $u$ is a solution of \eqref{eq:gKdV} such that $\lim_{t\to+\infty} \nh{u(t)-Q(\cdot-t)}=0$, then $u(t)=Q(\cdot-t)$ for $t$ large enough.
	\item For $p>5$ ($L^2$ supercritical case), solitons are unstable (see Grillakis, Shatah and Strauss \cite{gss} and Bona, Souganidis and Strauss \cite{bss}). In particular, the previous asymptotic classification result does not hold in this case. More precisely, we have:
\end{itemize}

\begin{theo}[\cite{combet}] \label{th:month}
Let $p>5$.
\begin{enumerate}[(i)]
\item There exists a one-parameter family ${(U^A)}_{A\in\R}$ of solutions of \eqref{eq:gKdV} such that, for all $A\in\R$, \[ \lim_{t\to+\infty} \nh{U^A(t,\cdot+t)-Q}=0, \] and if $A'\in\R$ satisfies $A'\neq A$, then $U^{A'}\neq U^A$.
\item Conversely, if $u$ is a solution of \eqref{eq:gKdV} such that $\lim_{t\to+\infty} \inf_{y\in\R} \nh{u(t)-Q(\cdot-y)}=0$, then there exist $A\in\R$, $t_0\in\R$ and $x_0\in\R$ such that $u(t)=U^A(t,\cdot-x_0)$ for $t\geq t_0$.
\end{enumerate}
\end{theo}

We recall that this result was an adaptation to \eqref{eq:gKdV} of previous works, concerning the nonlinear Schrödinger equation, of Duyckaerts and Merle \cite{duymerle} and Duyckaerts and Roudenko \cite{duy}. The purpose of this paper is to extend Theorem \ref{th:month} to multi-solitons.

\subsection{Multi-solitons}

Now, we focus on  multi-soliton solutions. Given $2N$ parameters defining $N\geq 2$ solitons with different speeds, \begin{equation} \label{parameters} 0<c_1<\cdots<c_N, \quad x_1,\ldots,x_N\in\R, \end{equation} we set \[ R_j(t) = R_{c_j,x_j}(t)\quad \m{and}\quad R(t)=\sum_{j=1}^N R_j(t), \] and we call multi-soliton a solution $u(t)$ of \eqref{eq:gKdV} such that \begin{equation} \label{problem} \nh{u(t)-R(t)} \vers 0 \quad \text{as} \quad t\to+\infty. \end{equation}

Let us recall known results on multi-solitons:
\begin{itemize}
\item For $p=2$ and $3$ (KdV and mKdV), multi-solitons (in a stronger sense) are well-known to exist for any set of parameters \eqref{parameters}, as a consequence of the inverse scattering method.
\item In the $L^2$ subcritical and critical cases, \emph{i.e.} for \eqref{eq:gKdV} with $p\leq 5$, Martel \cite{martel:Nsoliton} constructed multi-solitons for any set of parameters \eqref{parameters}. The proof in \cite{martel:Nsoliton} follows the strategy of Merle \cite{merle:kblowup} (compactness argument) and relies on monotonicity properties developed in \cite{martel:revisited} (see also \cite{martel:tsai}). Recall that Martel, Merle and Tsai \cite{martel:tsai} proved stability and asymptotic stability of a sum of $N$ solitons for large time for the subcritical case. A refined version of the stability result of \cite{martel:tsai} shows that, for a given set of parameters, there exists a \emph{unique} multi-soliton solution satisfying \eqref{problem}, see Theorem~1 in \cite{martel:Nsoliton}.
\item In the $L^2$ supercritical case, \emph{i.e.} in a situation where solitons are known to be unstable, Côte, Martel and Merle \cite{martel:Nsolitons} have recently proved the existence of at least \emph{one} multi-soliton solution for \eqref{eq:gKdV}:
\end{itemize}

\begin{theo}[\cite{martel:Nsolitons}] \label{th:cmm}
Let $p>5$ and $N\geq 2$. Let $0<c_1<\cdots<c_N$ and $x_1,\ldots,x_N \in \R$. There exist $T_0 \in\R$, $C,\sigma_0>0$, and a solution $\phib \in C([T_0,+\infty), H^1)$ of \eqref{eq:gKdV} such that \[ \forall t \in [T_0,+\infty), \quad \nh{\phib(t)-R(t)} \leq C e^{-\sigma_0^{3/2} t}. \]
\end{theo}

Recall that, with respect to \cite{martel:Nsoliton,martel:tsai}, the proof of Theorem \ref{th:cmm} relies on an additional topological argument to control the unstable nature of the solitons. Moreover, note that no uniqueness result is proved in \cite{martel:Nsolitons}, contrary to the subcritical and critical cases \cite{martel:Nsoliton}. In fact, the objective of this paper is to prove uniqueness up to $N$ parameters, as suggested by Theorem \ref{th:month}.

\subsection{Main result and outline of the paper}

The whole paper is devoted to prove the following theorem of existence and uniqueness of a family of multi-solitons for the supercritical \eqref{eq:gKdV} equation.

\begin{theo} \label{th:main}
Let $p>5$, $N\geq 2$, $0<c_1<\cdots<c_N$ and $x_1,\ldots,x_N \in \R$. Denote $\displaystyle R=\sum_{j=1}^N R_{c_j,x_j}$.
\begin{enumerate}
\item There exists an $N$-parameter family ${(\phib_{\an})}_{(\an)\in\R^N}$ of solutions of \eqref{eq:gKdV} such that, for all $(\an)\in\R^N$, \[ \lim_{t\to +\infty} \nh{\phib_{\an}(t)-R(t)}=0, \] and if $(A'_1,\ldots,A'_N)\neq (\an)$, then $\phib_{A'_1,\ldots,A'_N} \neq \phib_{\an}$.
\item Conversely, if $u$ is a solution of \eqref{eq:gKdV} such that $\lim_{t\to+\infty} \nh{u(t)-R(t)}=0$, then there exists $(\an)\in\R^N$ such that $u=\phib_{\an}$.
\end{enumerate}
\end{theo}

\begin{rem}
The convergence of $\phib_{\an}$ to $R$ in Theorem \ref{th:main} is actually exponential in time, as in Theorem \ref{th:cmm}. See the proof of Theorem \ref{th:main} at the beginning of Section \ref{sec:construction} for more details.
\end{rem}

\begin{rem} \label{rem:NLS}
For the nonlinear Schrödinger equation, the question of the classification of multi-solitons as in Theorem \ref{th:main} is open. In fact, even for subcritical and critical cases, no general uniqueness result has been proved yet (see general existence results in \cite{merle:kblowup,perelman,rss,martel:NLS,martel:Nsolitons}).
\end{rem}

The paper is organized as follows. In the next section, we briefly recall some well-known results on solitons, multi-solitons, and on the linearized equation. One of the most important facts about the linearized equation, also strongly used in \cite{martel:Nsolitons,combet}, is the determination by Pego and Weinstein \cite{pegoweinstein} of the spectrum of the linearized operator $\L$ around the soliton $Q(x-t)$: $\sigma(\L)\cap\R = \{-e_0,0,+e_0\}$ with $e_0>0$, and moreover $e_0$ and $-e_0$ are simple eigenvalues of $\L$ with eigenfunctions $Y^+$ and $Y^-$. Indeed, $Y^{\pm}$ allow to control the negative directions of the linearized energy around a soliton (see Lemma \ref{th:Zc}). Moreover, by a simple scaling argument, we determine eigenvalues of the linearized operator around $Q_{c_j}$: $\pm e_j=\pm c_j^{3/2}e_0$ are eigenvalues with eigenfunctions $Y_j^{\pm}$ (see Notation \ref{Rj} for precise definitions).

In Section \ref{sec:construction}, we construct the family $(\phib_{\an})$ described in Theorem \ref{th:main}. To do this, we first claim Proposition \ref{th:princ}, which is the new key point of the proof of the multi-existence result, and can be summarized as follows. \textit{Let $\phib$ be a multi-soliton given by Theorem \ref{th:cmm}, $j\in\unn$ and $A_j\in\R$. Then there exists a solution $u(t)$ of \eqref{eq:gKdV} such that \[ \nh{u(t)-\phib(t)-A_je^{-e_jt}Y_j^+(t)} \leq e^{-(e_j+\c)t}, \] for $t$ large and for some small $\c>0$.} This means that, similarly as in \cite{combet} for one soliton, we can perturb the multi-soliton $\phib$ locally around \emph{one} given soliton at the order $e^{-e_jt}$. Since $e_1<\cdots<e_N$, $\phib_{\an}$ has to be constructed by iteration, from $j=1$ to $j=N$. Indeed, it is not significant to perturb $\phib$ at order $e_j$ before order $e_{j-1}$, since $e_j>e_{j-1}+\c$. Similarly, it seems that there exists no simple way to compare $\phib_{\an}$ to $\phib$. Finally, to prove Proposition \ref{th:princ}, we rely on refinements of arguments developed in \cite{martel:Nsolitons}, in particular the topological argument to control the unstable directions.

In Section \ref{sec:classification}, we classify all multi-solitons in terms of the family previously constructed. Once again, it appears that the identification of the solution has to be done step by step (after an improvement of the convergence rate, as in \cite{combet}), from order $e_1$ to order $e_N$. In this section, we strongly use special monotonicity properties of \eqref{eq:gKdV}, in particular to prove that any multi-soliton converges exponentially (Section \ref{sec:expoc}). Such arguments are not known for the nonlinear Schrödinger equations.

Finally, recall that in the one soliton case for \eqref{eq:gKdV} \cite{combet}, a construction of a family of approximate solutions of the linearized equation and fixed point arguments were used (among other things), as in the one soliton case for the nonlinear Schrödinger equation \cite{duy}. For multi-solitons, since the construction of approximate solutions is not natural (because of the interactions between solitons), we propose in this paper an alternate approach based only on compactness and energy methods.

\section{Preliminary results}

\subsection{Notation and first properties of the solitons}

\begin{notation}
They are available in the whole paper.
\begin{enumerate}[(a)]
\item $(\cdot,\cdot)$ denotes the $L^2(\R)$ scalar product.
\item The Sobolev space $H^s$ is defined by $H^s(\R) = \{ u\in \mathcal{D}'(\R)\ |\ {(1+\xi^2)}^{s/2}\hat{u}(\xi) \in L^2(\R) \}$, and in particular $H^1(\R) = \{ u\in L^2(\R)\ |\ \nh{u}^2 = \nld{u}^2 +\nld{u'}^2 <+\infty \} \hookrightarrow L^{\infty}(\R)$.
\item We denote $\px v = v_x$ the partial derivative of $v$ with respect to $x$.
\item All numbers $C,K$ appearing in inequalities are real constants (with respect to the context) strictly positive, which may change in each step of an inequality.
\end{enumerate}
\end{notation}

\begin{claim} \label{th:solitons}
For all $c>0$, one has:
\begin{enumerate}[(i)]
\item $Q_c>0$, $Q_c$ is even, $Q_c$ is $C^{\infty}$, and $Q'_c(x)<0$ for all $x>0$.
\item For all $j\geq 0$, there exists $C_j>0$ such that $Q_c^{(j)}(x) \sim C_je^{-\sqrt{c}|x|}$ as $|x|\to +\infty$.\\ In particular, for all $j\geq 0$, there exists $C'_j>0$ such that $|Q_c^{(j)}(x)|\leq C'_j e^{-\sqrt{c}|x|}$ for all $x\in\R$.
\item $Q''_c+Q_c^p = cQ_c$.
\end{enumerate}
\end{claim}

\subsection{Linearized equation}

Let $c>0$.

\subsubsection{Linearized operator around $Q_c$}

The linearized equation appears if one considers a solution of \eqref{eq:gKdV} close to the soliton $Q_c(x-ct)$. More precisely, if $u_c(t,x)=Q_c(x-ct)+h_c(t,x-ct)$ satisfies \eqref{eq:gKdV}, then $h_c$ satisfies \[ \partial_t h_c+\L_c h_c=O(h_c^2) \] where \[ \L_c a=-\px(L_ca) \quad \m{and}\quad L_ca=-\px^2 a+ca-pQ_c^{p-1}a. \] The spectrum of $\L_c$ has been calculated by Pego and Weinstein for $c=1$ in \cite{pegoweinstein}. Their results are summed up in the following proposition for the reader's convenience.

\begin{prop}[\cite{pegoweinstein}]
Let $\sigma(\L)$ be the spectrum of the operator $\L$ defined on $L^2(\R)$ and let $\sigma_{\mathrm{ess}}(\L)$ be its essential spectrum. Then \[\sigma_{\mathrm{ess}}(\L)=i\R \quad \m{and}\quad \sigma(\L)\cap\R = \{-e_0,0,e_0\} \m{ with } e_0>0. \] Furthermore, $e_0$ and $-e_0$ are simple eigenvalues of $\L$ with eigenfunctions $Y^+$ and $Y^- = \check{Y}^+$ which have an exponential decay at infinity, and the null space of $\L$ is spanned by $Q'$.
\end{prop}

This result is extended to $\L_c$ in Corollary~\ref{th:spectrumc} by a simple scaling argument. Indeed, we recall that if $u$ is a solution of \eqref{eq:gKdV}, then for all $\lambda>0$, $u_{\lambda}(t,x)=\lambda^{\frac{2}{p-1}}u(\lambda^3 t,\lambda x)$ is also a solution. Moreover, we have $Q_c(x)=c^{\frac{1}{p-1}}Q(\sqrt cx)$.

\begin{cor} \label{th:spectrumc}
Let $\sigma(\L_c)$ be the spectrum of the operator $\L_c$ defined on $L^2(\R)$ and let $\sigma_{\mathrm{ess}}(\L_c)$ be its essential spectrum. Then \[\sigma_{\mathrm{ess}}(\L_c)=i\R \quad \m{and}\quad \sigma(\L_c)\cap\R = \{-e_c,0,e_c\}\ \m{ where }\ e_c = c^{3/2}e_0>0. \] Furthermore, $e_c$ and $-e_c$ are simple eigenvalues of $\L_c$ with eigenfunctions $Y_c^+$ and $Y_c^- = \check{Y_c^+}$, where \[ Y^{\pm}_c(x)=c^{-1/2} Y^{\pm}(\sqrt cx), \] and the null space of $\L_c$ is spanned by $Q'_c$.
\end{cor}

\subsubsection{Adjoint of $\L_c$}

We recall that Lemma 4.9 in \cite{combet}, under a suitable normalization of $Y^{\pm}$, shows important properties of the adjoint of $\L$. With the same normalization and by Corollary \ref{th:spectrumc}, we obtain the following lemma by a simple scaling argument. Recall that assertion (v) is proved in \cite{martel:Nsolitons} for $c=1$.

\begin{lem} \label{th:Zc}
Let $Z_c^{\pm} = L_cY_c^{\pm}$. Then the following properties hold:
\begin{enumerate}[(i)]
\item $Z_c^{\pm}$ are two eigenfunctions of $L_c\px$: $L_c(\px Z_c^{\pm}) = \mp e_c Z_c^{\pm}$.
\item There exists $\eta_0>0$ such that, for all $x\in\R$, \[ |Y_c^{\pm}(x)| + |\px Y_c^{\pm}(x)| + |Z_c^{\pm}(x)| + |\px Z_c^{\pm}(x)| \leq Ce^{-\eta_0\sqrt{c}|x|}. \]
\item $(Y_c^+,Z_c^+)=(Y_c^-,Z_c^-)=0$ and $(Z_c^+,Q'_c)=(Z_c^-,Q'_c)=0$.
\item $(Y_c^+,Z_c^-)=(Y_c^-,Z_c^+)=1$ and $(Q'_c,\px Y_c^+)>0$.
\item There exists $\widetilde{\sigma_c}>0$ such that, for all $v_c\in H^1$ such that $(v_c,Z_c^+)=(v_c,Z_c^-)=(v_c,Q'_c)=0$, $(L_cv_c,v_c)\geq \widetilde{\sigma_c} \nh{v_c}^2$.
\item There exist $\sigma_c>0$ and $C>0$ such that, for all $v_c\in H^1$, \[ (L_cv_c,v_c)\geq \sigma_c \nh{v_c}^2 -C\carre{(v_c,Z_c^+)} -C\carre{(v_c,Z_c^-)} -C\carre{(v_c,Q'_c)}. \]
\end{enumerate}
\end{lem}

\subsection{Multi-solitons results}

A set of parameters \eqref{parameters} being given, we adopt the following notation.

\begin{notation} \label{Rj}
For all $j\in \unn$, define:
\begin{enumerate}[(i)]
\item $R_j(t,x)=Q_{c_j}(x-c_jt-x_j)$, where $Q_c(x)=c^{\frac{1}{p-1}}Q(\sqrt cx)$.
\item $Y_j^{\pm}(t,x)=Y_{c_j}^{\pm}(x-c_jt-x_j)$, where $Y_c^{\pm}(x)=c^{-1/2}Y^{\pm}(\sqrt cx)$ is defined in Corollary \ref{th:spectrumc}.
\item $Z_j^{\pm}(t,x)=Z_{c_j}^{\pm}(x-c_jt-x_j)$, where $Z_c^{\pm}=L_cY_c^{\pm}$.
\item $e_j=e_{c_j}$, where $e_c=c^{3/2}e_0$.
\end{enumerate}
\end{notation}

Now, to estimate interactions between solitons, we denote the small parameters \begin{equation} \label{gamma} \sigma_0 = \min \{ \eta_0^{2/3}c_1,e_0^{2/3}c_1,c_1,c_2-c_1,\ldots,c_N-c_{N-1} \} \quad \m{and} \quad  \c = \frac{\sigma_0^{3/2}}{10^6}. \end{equation}

From \cite{martel:Nsoliton}, it appears that $\c$ is a suitable parameter to quantify interactions between solitons in large time. For instance, we have, for $j\neq k$ and all $t\geq 0$, \begin{equation} \label{eq:interact} \int R_j(t)R_k(t) + |(R_j)_x(t)||(R_k)_x(t)| \leq Ce^{-10\c t}. \end{equation} From the definition of $\sigma_0$ and Lemma \ref{th:Zc}, such an inequality is also true for $Y_j^{\pm}$ and $Z_j^{\pm}$.

Moreover, since $\sigma_0$ has the same definition as in \cite{martel:Nsolitons}, then from their Remark 1, Theorem \ref{th:cmm} can be rewritten as follows. There exist $T_0\in\R$ and $\phib \in C([T_0,+\infty),H^1)$ such that, for all $s\geq 1$, there exists $A_s>0$ such that \begin{equation} \label{eq:cmm} \nhs{\phib(t)-R(t)} \leq A_s e^{-4\c t}. \end{equation}

\section{Construction of a family of multi-solitons} \label{sec:construction}

In this section, we prove the first point of Theorem \ref{th:main} as a consequence of the following crucial Proposition \ref{th:princ}. Let $p>5$, $N\geq 2$, $0<c_1<\cdots<c_N$ and $x_1,\ldots,x_N \in \R$. Denote $R=\sum_{k=1}^N R_k$ and $\phib$ a multi-soliton solution satisfying \eqref{eq:cmm}, as defined in Theorem \ref{th:cmm} for example.

\begin{prop} \label{th:princ}
Let $j\in \unn$ and $A_j\in\R$. Then there exist $t_0>0$ and $u \in C([t_0,+\infty),H^1)$ a solution of \eqref{eq:gKdV} such that \begin{equation} \label{eq:perturb} \forall t\geq t_0,\quad \nh{u(t)-\phib(t)-A_je^{-e_jt}Y_j^+(t)} \leq e^{-(e_j+\c)t}. \end{equation}
\end{prop}

Before proving this proposition, let us show how this proposition implies the first point of Theorem \ref{th:main}.

\begin{proof}[Proof of 1. of Theorem \ref{th:main}]
Let $(\an)\in\R^N$.
\begin{enumerate}[(i)]
\item Consider $\phib_{A_1}$ the solution of \eqref{eq:gKdV} given by Proposition \ref{th:princ} applied with $\phib$ given by Theorem \ref{th:cmm}. Thus there exists $t_0>0$ such that \[ \forall t\geq t_0,\quad \nh{\phib_{A_1}(t)-\phib(t)-A_1e^{-e_1 t}Y_1^+(t)} \leq e^{-(e_1+\c)t}. \] Now remark that $\phib_{A_1}$ is also a multi-soliton, which satisfies \eqref{eq:cmm} by the definition of $\c$ and the same techniques used in \cite[Section 3.4]{martel:Nsoliton} to improve the estimate in higher order Sobolev norms. Hence we can apply Proposition \ref{th:princ} with $\phib_{A_1}$ instead of $\phib$, so that we obtain $\phib_{A_1,A_2}$ such that \[ \forall t\geq t'_0,\quad \nh{\phib_{A_1,A_2}(t)-\phib_{A_1}(t)-A_2e^{-e_2t}Y_2^+(t)} \leq e^{-(e_2+\c)t}. \] Similarly, for all $j\in\unn$, we construct by induction a solution $\phib_{\aj}$ such that \begin{equation} \label{eq:phiaj} \forall t\geq t_0,\quad \nh{\phib_{\aj}(t)-\phib_{A_1,\ldots,A_{j-1}}(t)-A_je^{-e_jt}Y_j^+(t)} \leq e^{-(e_j+\c)t}. \end{equation} Observe finally that $\phib_{\an}$ constructed by this way satisfies \eqref{eq:cmm}.
\item Let $(A'_1,\ldots,A'_N)\in\R^N$ be such that $(A'_1,\ldots,A'_N)\neq (\an)$, and suppose in the sake of contradiction that $\phib_{A'_1,\ldots,A'_N} = \phib_{\an}$. Denote $i_0 = \min\{ i\in\unn\ |\ A'_i\neq A_i \}$. Hence we have $A'_i=A_i$ for $i\in \ient{1}{i_0-1}$, $A'_{i_0}\neq A_{i_0}$ and from the construction of $\phib_{\an}$, \begin{align*} \phib_{\an}(t) &= \phib_{A_1,\ldots,A_{N-1}}(t) +A_Ne^{-e_Nt}Y_N^+(t) +z_N(t)\\ &= \phib_{A_1,\ldots,A_{N-2}}(t) +A_{N-1}e^{-e_{N-1}t}Y_{N-1}^+(t)+A_Ne^{-e_Nt}Y_N^+(t) +z_{N-1}(t)+z_N(t)\\ &= \cdots = \phib_{A_1,\ldots,A_{i_0-1}}(t) +A_{i_0}e^{-e_{i_0}t}Y_{i_0}^+(t) +\sum_{k>i_0} A_k e^{-e_kt}Y_k^+(t) +\sum_{k\geq i_0} z_k(t) \end{align*} where $z_k$ satisfies $\nh{z_k(t)}\leq e^{-(e_k+\c)t}$ for $t\geq t_0$ and each $k\geq i_0$. Similarly, we get \[ \phib_{A'_1,\ldots,A'_N}(t) = \phib_{A'_1,\ldots,A'_{i_0-1}}(t) +A'_{i_0}e^{-e_{i_0}t}Y_{i_0}^+(t) +\sum_{k>i_0} A'_k e^{-e_kt}Y_k^+(t) +\sum_{k\geq i_0} \widetilde{z_k}(t), \] and so using $\phib_{A'_1,\ldots,A'_N} = \phib_{\an}$ and $\phib_{A'_1,\ldots,A'_{i_0-1}} = \phib_{A_1,\ldots,A_{i_0-1}}$, we obtain \[ e^{-e_{i_0}t}|A_{i_0}-A'_{i_0}|\leq Ce^{-(e_{i_0}+\c)t} \] for $t\geq t_0$, thus $|A_{i_0}-A'_{i_0}|\leq Ce^{-\c t}$, and so $A'_{i_0}=A_{i_0}$ by letting $t\to +\infty$, which is a contradiction and concludes the proof. \qedhere
\end{enumerate}
\end{proof}

Now, the only purpose of the rest of this section is to prove Proposition \ref{th:princ}. Let $j\in \unn$ and $A_j\in\R$. We want to construct a solution $u$ of \eqref{eq:gKdV} such that \[ z(t,x)= u(t,x)-\phib(t,x)-A_je^{-e_jt}Y_j^+(t,x) \] satisfies $\nh{z(t)}\leq e^{-(e_j+\c)t}$ for $t\geq t_0$ with $t_0$ large enough.

\subsection{Equation of $z$}

Since $u$ is a solution of \eqref{eq:gKdV} and also $\phib$ is (and this fact is crucial for the whole proof), we get \[ \partial_t z+\px^3 z +\px[{(\phib+A_je^{-e_jt}Y_j^+ +z)}^p-\phib^p] +A_je^{-e_jt}[\px^3 Y_j^+ -c_j\px Y_j^+ -e_jY_j^+] =0. \] But from Corollary \ref{th:spectrumc}, we have \[ \L_{c_j}Y_{c_j}^+ = e_jY_{c_j}^+ = \px^3 Y_{c_j}^+ -c_j\px Y_{c_j}^+ +p\px(Q_{c_j}^{p-1}Y_{c_j}^+) \] and so following Notation \ref{Rj}, we get the following equation for $z$: \begin{equation} \label{eq:z1} \partial_t z +\px^3 z +\px[{(\phib+A_je^{-e_jt}Y_j^+ +z)}^p -\phib^p -pA_je^{-e_jt}R_j^{p-1}Y_j^+] = 0. \end{equation} This can also be written \begin{multline*} \partial_t z + \px\big[\px^2 z+ p\phib^{p-1}z\big] + \px\big[{(\phib+A_je^{-e_jt}Y_j^++z)}^p -{(\phib+A_je^{-e_jt}Y_j^+)}^p -p{(\phib+A_je^{-e_jt}Y_j^+)}^{p-1}z\big]\\ +p\px\big[({(\phib+A_je^{-e_jt}Y_j^+)}^{p-1}-\phib^{p-1})\cdot z\big] = -\px\big[{(\phib+A_je^{-e_jt}Y_j^+)}^p -\phib^p -pA_je^{-e_jt}Y_j^+ R_j^{p-1}\big]. \end{multline*}

Finally, if we denote \[ \left\{ \begin{aligned} \omega_1 &= p[{(\phib+A_je^{-e_jt}Y_j^+)}^{p-1}-\phib^{p-1}],\\ \omega(z) &= {(\phib+A_je^{-e_jt}Y_j^++z)}^p -{(\phib+A_je^{-e_jt}Y_j^+)}^p -p{(\phib+A_je^{-e_jt}Y_j^+)}^{p-1}z,\\ \Omega &= {(\phib+A_je^{-e_jt}Y_j^+)}^p -\phib^p -pA_je^{-e_jt}Y_j^+ R_j^{p-1}, \end{aligned} \right. \] we obtain the shorter form of the equation of $z$: \begin{equation} \label{eq:z2} \partial_t z + \px\big[\px^2 z+ p\phib^{p-1}z\big] + \px[\omega_1\cdot z] +\px[\omega(z)] = -\px\Omega. \end{equation}

Note that the term $\omega(z)$ is the nonlinear term in $z$, and that $\omega_1$ satisfies, for all $s\geq 0$, ${\|\omega_1(t)\|}_{H^s}\leq C_se^{-e_jt}$ for all $t\geq 0$. Moreover, the source term $\Omega$ satisfies \begin{equation} \label{Omega} \forall s\geq 1, \exists C_s>0, \forall t\geq 0,\quad \nhs{\Omega(t)} \leq C_se^{-(e_j+4\c)t}. \end{equation} Indeed, if we write $\Omega$ under the form \begin{multline*} \Omega = \big[ {(\phib+A_je^{-e_jt}Y_j^+)}^p-\phib^p-p\phib^{p-1}A_je^{-e_jt}Y_j^+ \big]\\ +pA_je^{-e_jt}Y_j^+(\phib^{p-1}-R^{p-1}) +pA_je^{-e_jt}Y_j^+(R^{p-1}-R_j^{p-1}), \end{multline*} we deduce from \eqref{eq:cmm}, \eqref{eq:interact} and the definition of $\c$ \eqref{gamma} that \[ \nhs{\Omega(t)} \leq Ce^{-2e_jt} +Ce^{-e_jt}\nhs{\phib(t)-R(t)} +Ce^{-e_jt}\cdot e^{-4\c t} \leq Ce^{-(e_j+4\c)t}. \]

\subsection{Compactness argument assuming uniform estimate}

To prove Proposition \ref{th:princ}, we follow the strategy of \cite{martel:Nsoliton,martel:Nsolitons}. Let $S_n\to+\infty$ be an increasing sequence of time, $\b_n={(b_{n,k})}_{j<k\leq N} \in\R^{N-j}$ be a sequence of parameters to be determined, and let $u_n$ be the solution of \begin{equation} \label{un} \begin{cases} \partial_t u_n + \px[\px^2 u_n +u_n^p]=0,\\ \displaystyle u_n(S_n) = \phib(S_n) +A_je^{-e_jS_n}Y_j^+(S_n)+\sum_{k>j} b_{n,k}Y_k^+(S_n). \end{cases} \end{equation}

\begin{notation}
\begin{enumerate}[(i)]
\item $\R^N$ is equipped with the $\ell^2$ norm, simply denoted $\|\cdot\|$.
\item $B_{\mathcal{B}}(P,r)$ is the closed ball of the Banach space $\mathcal{B}$, centered at $P$ and of radius $r\geq 0$. If $P=0$, we simply write $B_{\mathcal{B}}(r)$.
\item $\S_{\R^N}(r)$ denotes the sphere of radius $r$ in $\R^N$.
\end{enumerate}
\end{notation}

\begin{prop} \label{th:princbis}
There exist $n_0\geq 0$ and $t_0>0$ (independent of $n$) such that the following holds. For each $n\geq n_0$, there exists $\b_n\in\R^{N-j}$ with $\| \b_n\| \leq 2e^{-(e_j+2\c)S_n}$, and such that the solution $u_n$ of \eqref{un} is defined on the interval $[t_0,S_n]$, and satisfies \[ \forall t\in [t_0,S_n],\quad \nh{u_n(t)-\phib(t) -A_je^{-e_jt}Y_j^+(t)} \leq e^{-(e_j+\c)t}. \]
\end{prop}

Assuming this proposition and the following lemma of weak continuity of the flow, we can deduce the proof of Proposition \ref{th:princ}. The proof of Proposition \ref{th:princbis} is postponed to the next section, whereas the proof of Lemma \ref{th:wcf} is postponed to Appendix \ref{appendix}.

\begin{lem} \label{th:wcf}
Suppose that $z_{0,n}\cvf z_0$ in $H^1$, and that there exits $T>0$ such that the solution $z_n(t)$ corresponding to initial data $z_{0,n}$ exists for $t\in [0,T]$ and $\supt \nh{z_n(t)} \leq K$. Then for all $t\in [0,T]$, the solution $z(t)$ corresponding to initial data $z_0$ exists, and $z_n(T)\cvf z(T)$ in $H^1$.
\end{lem}

\begin{rem}
Note that the proof of Lemma \ref{th:wcf} strongly relies on the Cauchy theory in $H^s$ with $s<1$, developed in \cite{kpv}. Thus this argument is quite similar to the compactness argument developed in \cite{martel:Nsolitons} or \cite{martel:Nsoliton}.
\end{rem}

\begin{proof}[Proof of Proposition \ref{th:princ} assuming Proposition \ref{th:princbis}]
We may assume $n_0=0$ in Proposition \ref{th:princbis} without loss of generality. It follows from this proposition that there exists a sequence $u_n(t)$ of solutions to \eqref{eq:gKdV}, defined on $[t_0,S_n]$, such that the following uniform estimates hold: \[ \forall n\geq 0, \forall t\in [t_0,S_n],\quad \nh{u_n(t)-\phib(t)-A_je^{-e_jt}Y_j^+(t)}\leq e^{-(e_j+\c)t}. \] In particular, there exists $C_0>0$ such that $\nh{u_n(t_0)}\leq C_0$ for all $n\geq 0$. Thus there exists $u_0\in H^1(\R)$ such that $u_n(t_0)\cvf u_0$ in $H^1$ weak (after passing to a subsequence). Now consider $u$ solution of \[ \begin{cases} \partial_t u +\px[\px^2 u+u^p]=0,\\ u(t_0)=u_0. \end{cases} \] Let $T\geq t_0$. For $n$ such that $S_n>T$, $u_n(t)$ is well defined for all $t\in [t_0,T]$, and moreover $\nh{u_n(t)}\leq C$. By Lemma \ref{th:wcf}, we have $u_n(T)\cvf u(T)$ in $H^1$. As \[ \nh{u_n(T)-\phib(T)-A_je^{-e_jT}Y_j^+(T)} \leq e^{-(e_j+\c)T}, \] we finally obtain, by weak convergence, $\nh{u(T)-\phib(T)-A_je^{-e_jT}Y_j^+(T)}\leq e^{-(e_j+\c)T}$. Thus $u$ is a solution of \eqref{eq:gKdV} which satisfies \eqref{eq:perturb}.
\end{proof}

\subsection{Proof of Proposition \ref{th:princbis}} \label{sec:preuveprinc}

The proof proceeds in several steps. For the sake of simplicity, we will drop the index $n$  for the rest of this section (except for $S_n$). As Proposition \ref{th:princbis} is proved for given $n$, this should not be a source of confusion. Hence we will write $u$ for $u_n$, $z$ for $z_n$, $\b$ for $\b_n$, etc. We possibly drop the first terms of the sequence $S_n$, so that, for all $n$, $S_n$ is large enough for our purposes.

From \eqref{eq:z2}, the equation satisfied by $z$ is \begin{equation} \begin{cases} \label{eq:z3} \partial_t z+\px[\px^2 z+p\phib^{p-1}z] + \px[\omega_1\cdot z] +\px[\omega(z)] = -\px\Omega,\\ z(S_n) = \sum_{k>j} b_k Y_k^+(S_n). \end{cases} \end{equation} Moreover, for all $k\in\unn$, we denote \[ \alpha_k^{\pm}(t) = \int z(t)\cdot Z_k^{\pm}(t). \] In particular, we have \[ \alpha_k^{\pm}(S_n) = \sum_{l>j} b_l \int Y_l^+(S_n)\cdot Z_k^{\pm}(S_n). \] Finally, we denote $\alpham(t)={(\alpha_k^-(t))}_{j<k\leq N}$.

\subsubsection{Modulated final data}

\begin{lem} \label{th:finaldata}
For $n\geq n_0$ large enough, the following holds. For all $\a\in\R^{N-j}$, there exists a unique $\b\in\R^{N-j}$ such that $\|\b\| \leq 2\|\a\|$ and $\alpham(S_n)=\a$.
\end{lem}

\begin{proof}
Consider the linear application \[ \begin{array}{rrcl} \Phi~: &\R^{N-j} &\to &\R^{N-j}\\ &\b={(b_l)}_{j<l\leq N} &\mapsto &{\left( \sum_{l>j} b_l \int Y_l^+(S_n)Z_k^-(S_n) \right)}_{j<k\leq N}. \end{array} \] From the normalization of Lemma \ref{th:Zc}, its matrix in the canonical basis is \[ \mathrm{Mat}\,\Phi = \begin{pmatrix} 1& \int Y_{j+2}^+ Z_{j+1}^-(S_n)&\cdots& \int Y_{j+N}^+ Z_{j+1}^-(S_n) \\ \int Y_{j+1}^+ Z_{j+2}^-(S_n)&1&\cdots&\vdots \\ \vdots&\vdots&\ddots&\vdots \\ \int Y_{j+1}^+ Z_{j+N}^-(S_n)&\cdots&\cdots&1 \end{pmatrix}. \] But from \eqref{eq:interact}, we have, for $k\neq l$, \[ \left| \int Y_l^{\pm}Z_k^{\pm}(S_n) \right| \leq C_0e^{-\c S_n} \] with $C_0$ independent of $n$, and so by taking $n_0$ large enough, we have $\Phi = \mathrm{Id} +A_n$ where $\|A_n\|\leq \frac{1}{2}$. Thus $\Phi$ is invertible and $\|\Phi^{-1}\|\leq 2$. Finally, for a given $\a\in \R^{N-j}$, it is enough to define $\b$ by $\b=\Phi^{-1}(\a)$ to conclude the proof of Lemma \ref{th:finaldata}.
\end{proof}

\begin{claim} \label{Sn}
The following estimates at $S_n$ hold:
\begin{itemize}
\item $|\alpha_k^+(S_n)|\leq Ce^{-2\c S_n}\|\b\|$ for all $k\in\unn$,
\item $|\alpha_k^-(S_n)|\leq Ce^{-2\c S_n}\|\b\|$ for all $k\in \ient{1}{j}$,
\item $\nh{z(S_n)}\leq C\|\b\|$.
\end{itemize}
\end{claim}

\subsubsection{Equations on $\alpha_k^{\pm}$}

Let $t_0>0$ independent of $n$ to be determined later in the proof, $\a\in B_{\R^{N-j}}(e^{-(e_j+2\c)S_n})$ to be chosen, $\b$ be given by Lemma \ref{th:finaldata} and $u$ be the corresponding solution of \eqref{un}. We now define the maximal time interval $[\ta,S_n]$ on which suitable exponential estimates hold.

\begin{defi} \label{def:ta}
Let $\ta$ be the infimum of $T\geq t_0$ such that for all $t\in [T,S_n]$, both following properties hold: \begin{equation} \label{eq:ta} e^{(e_j+\c)t}z(t) \in B_{H^1}(1) \quad \m{and} \quad e^{(e_j+2\c)t} \alpham(t) \in B_{\R^{N-j}}(1). \end{equation}
\end{defi}

Observe that Proposition \ref{th:princbis} is proved if for all $n$, we can find $\a$ such that $\ta = t_0$. The rest of the proof is devoted to prove the existence of such a value of $\a$.

\bigskip

First, we prove the following estimate on $\alpha_k^{\pm}$.

\begin{claim}
For all $k\in\unn$ and all $t\in [\ta,S_n]$, \begin{equation} \label{eq:alpha} \left|\dt \alpha_k^{\pm}(t) \mp e_k\alpha_k^{\pm}(t) \right| \leq C_0e^{-4\c t}\nh{z(t)} +C_1\nh{z(t)}^2 +C_2e^{-(e_j+4\c)t}. \end{equation}
\end{claim}

\begin{proof}
Using the equation of $z$ \eqref{eq:z3}, we first compute \begin{align*} \dt \alpha_k^{\pm}(t) &= \int z_t Z_k^{\pm} +\int zZ_{kt}^{\pm}\\ &= \int (z_{xx}+p\phib^{p-1}z)Z_{kx}^{\pm} +\int \omega_1zZ_{kx}^{\pm} + \int \omega(z)Z_{kx}^{\pm} + \int \Omega Z_{kx}^{\pm} -c_k\int zZ_{kx}^{\pm}\\ &= \int (z_{xx}-c_kz+pR_k^{p-1}z)Z_{kx}^{\pm} +p\int (\phib^{p-1}-R_k^{p-1})zZ_{kx}^{\pm} +\int (\omega_1z+\omega(z)+\Omega)Z_{kx}^{\pm}. \end{align*} But from (i) of Lemma \ref{th:Zc}, we have \begin{align*} \int (z_{xx} -c_kz+pR_k^{p-1}z)Z_{kx}^{\pm} &= (-L_{c_k}z(t,\cdot+c_kt),\px Z_{c_k}^{\pm})\\ &= (z(t,\cdot+c_kt),-L_{c_k}(\px Z_{c_k}^{\pm})) = \pm e_k(z(t,\cdot+c_kt),Z_{c_k}^{\pm}) = \pm e_k \alpha_k^{\pm}, \end{align*} and from \eqref{eq:cmm} and \eqref{Omega}, we have the following estimates:
\begin{itemize}
\item $|\int (\phib^{p-1}-R_k^{p-1})zZ_{kx}^{\pm}|\leq C\nli{\phib-R}\nli{z} +Ce^{-4\c t}\nld{z} \leq Ce^{-4\c t}\nh{z}$,
\item $|\int \omega_1zZ_{kx}^{\pm}| \leq \nli{\omega_1}\nli{z}\nlu{Z_{kx}^{\pm}} \leq Ce^{-e_jt}\nh{z} \leq Ce^{-4\c t}\nh{z}$,
\item $|\int \omega(z)Z_{kx}^{\pm}|\leq C\nld{z}^2 \leq C\nh{z}^2$,
\item $|\int \Omega Z_{kx}^{\pm}| \leq C\nli{\Omega} \leq Ce^{-(e_j+4\c)t}$,
\end{itemize}
which conclude the proof of the claim.
\end{proof}

\subsubsection{Control of the stable directions}

We estimate here $\alpha_k^+(t)$ for all $k\in\unn$ and $t\in [\ta,S_n]$. From \eqref{eq:alpha} and \eqref{eq:ta}, we have \[ \left| \dt \alpha_k^+(t) -e_k\alpha_k^+(t)\right| \leq C_0e^{-(e_j+5\c)t} +C_1e^{-2(e_j+\c)t} +C_2e^{-(e_j+4\c)t} \leq K_2e^{-(e_j+4\c)t}. \] Thus $|{(e^{-e_ks}\alpha_k^+(s))}'|\leq K_2e^{-(e_j+e_k+4\c)s}$, and so by integration on $[t,S_n]$ we get $|e^{-e_kS_n}\alpha_k^+(S_n)-e^{-e_kt}\alpha_k^+(t)| \leq K_2e^{-(e_j+e_k+4\c)t}$ and so \[ |\alpha_k^+(t)|\leq e^{e_k(t-S_n)}|\alpha_k^+(S_n)|+K_2e^{-(e_j+4\c)t}. \] But from Claim \ref{Sn} and Lemma \ref{th:finaldata}, we have \begin{align*} e^{e_k(t-S_n)}|\alpha_k^+(S_n)| \leq |\alpha_k^+(S_n)| &\leq Ce^{-2\c S_n}\|\b\|\\ &\leq Ce^{-2\c S_n} e^{-(e_j+2\c)S_n} \leq K_2e^{-(e_j+4\c)S_n} \leq K_2e^{-(e_j+4\c)t}, \end{align*} and so finally \begin{equation} \label{eq:alphap} \forall k\in\unn, \forall t\in[\ta,S_n],\quad |\alpha_k^+(t)|\leq K_2e^{-(e_j+4\c)t}. \end{equation}

\subsubsection{Control of the unstable directions for $k\leq j$}

We estimate here $\alpha_k^-(t)$ for all $k\in\ient{1}{j}$ and $t\in [\ta,S_n]$. Note first that, as in the previous paragraph, we get for all $k\in\unn$ and $t\in [\ta,S_n]$, \begin{equation} \label{eq:alpham} \left| \dt \alpha_k^-(t)+e_k\alpha_k^-(t)\right| \leq K_2e^{-(e_j+4\c)t}. \end{equation} Now suppose $k\leq j$, which implies $e_k\leq e_j$. Since $|{(e^{e_ks}\alpha_k^-(s))}'|\leq K_2e^{(e_k-e_j-4\c)s}$, we obtain, by integration on $[t,S_n]$, \[ |\alpha_k^-(t)|\leq e^{e_k(S_n-t)}|\alpha_k^-(S_n)| +K_2e^{-(e_j+4\c)t}. \] But again from Claim \ref{Sn} and Lemma \ref{th:finaldata}, we have \begin{align*} e^{e_k(S_n-t)}|\alpha_k^-(S_n)| &\leq K_2 e^{e_k(S_n-t)}e^{-2\c S_n}e^{-(e_j+2\c)S_n} = K_2 e^{e_k(S_n-t)}e^{-(e_j+4\c)S_n}\\ &\leq K_2e^{(S_n-t)(e_k-e_j)}e^{-e_jt}e^{-4\c S_n} \leq K_2e^{-(e_j+4\c)t}, \end{align*} and so finally \begin{equation} \label{eq:alphambis} \forall k\in\ient{1}{j}, \forall t\in[\ta,S_n],\quad |\alpha_k^-(t)|\leq K_2e^{-(e_j+4\c)t}. \end{equation}

\subsubsection{Monotonicity property of the energy} \label{sec:monotonicity}

We follow here the same strategy as in \cite[Section 4]{martel:Nsoliton} to estimate the energy backwards. Since calculations are long and technical, we refer to \cite{martel:Nsoliton} for more details.

We define the following function \[ \psi(x)=\frac{2}{\pi}\arctan(\exp(-\sqrt{\sigma_0}x/2)) \] so that $\lim_{+\infty} \psi=0$, $\lim_{-\infty} \psi=1$, and for all $x\in\R$, $\psi(-x)=1-\psi(x)$. Note that by a direct calculation, we have $|\psi'''(x)|\leq \frac{\sigma_0}{4}|\psi'(x)|$. Moreover, we set \[ h(t,x) = \frac{1}{c_N} + \sum_{k=1}^{N-1} \left( \frac{1}{c_k}-\frac{1}{c_{k+1}}\right) \psi\left(x- \frac{c_k+c_{k+1}}{2}t -\frac{x_k+x_{k+1}}{2} \right). \] Observe that the function $h$ takes values close to $\frac{1}{c_k}$ for $x$ close to $c_kt+x_k$, and has large variations only in regions far away from the solitons (for instance we have, for all $k\in\unn$ and $t\geq 0$, $\nli{R_k(t)h_x(t)}\leq Ce^{-4\c t}$). We also define a quantity related to the energy for $z$: \[ H(t) = \int \left\{ \Big( z_x^2(t,x)-F(t,z(t,x))\Big)h(t,x)+z^2(t,x)\right\} dx \] where \[ F(t,z) = 2\left[ \frac{{(\phib+v_j+z)}^{p+1}}{p+1}-\frac{{(\phib+v_j)}^{p+1}}{p+1} -{(\phib+v_j)}^pz \right], \] and $v_j(t,x)=A_je^{-e_jt}Y_j^+(t,x)$.

\begin{lem} \label{th:dHdt}
For all $t\in [\ta,S_n]$, \[ \frac{dH}{dt}(t) \geq -C_0\nh{z(t)}^3 -C_1e^{-2\c t}\nh{z(t)}^2 -C_2 e^{-(e_j+3\c)t}\nh{z(t)}. \]
\end{lem}

\begin{proof}
Since $\frac{\partial F}{\partial z} = 2[{(\phib+v_j+z)}^p-{(\phib+v_j)}^p]$, we can first compute \begin{multline*} \frac{dH}{dt} = \int (z_x^2-F(z))h_t -2\int z_t\big[{(\phib+v_j+z)}^p-{(\phib+v_j)}^p]h +2\int z_{xt}z_xh +2\int z_tz\\ -2\int {(\phib+v_j)}_t\Big[{(\phib+v_j+z)}^p-{(\phib+v_j)}^p-p{(\phib+v_j)}^{p-1}z\Big]h. \end{multline*} Moreover $2\int z_{xt}z_xh = -2\int z_t(z_{xx}h+z_xh_x)$, thus \begin{multline*} \frac{dH}{dt} = \int (z_x^2-F(z))h_t -2\int z_t\big[z_{xx}+{(\phib+v_j+z)}^p-{(\phib+v_j)}^p]h +2\int z_t(z-z_xh_x)\\ -2\int {(\phib+v_j)}_t\Big[{(\phib+v_j+z)}^p-{(\phib+v_j)}^p-p{(\phib+v_j)}^{p-1}z\Big]h. \end{multline*} Now we replace $z_t$ thanks to the equation that it satisfies, which can be written, from \eqref{eq:z1}, \[ z_t +{\Big[ z_{xx} +{(\phib+v_j+z)}^p -{(\phib+v_j)}^p \Big]}_x = -\Omega_x. \] Using multiple integrations by parts, we finally obtain \begin{align} \frac{dH}{dt} &= \int (z_x^2-F(z))h_t +\int z_x^2 h_{xxx} \label{eq:dh1}\\ &+2\int z_xh_x{\Big[{(\phib+v_j+z)}^p -{(\phib+v_j)}^p \Big]}_x \label{eq:dh2}\\ &-2\int z{\Big[{(\phib+v_j+z)}^p-{(\phib+v_j)}^p\Big]}_x -2\int \phib_th\Big[ {(\phib+v_j+z)}^p -{(\phib+v_j)}^p -p{(\phib+v_j)}^{p-1}z\Big] \label{eq:dh3}\\ &-2\int z\Omega_x +2\int zh\Omega_{xxx} +2\int zh_x\Omega_{xx} +2\int h\Omega_x\Big[ {(\phib+v_j+z)}^p-{(\phib+v_j)}^p\Big] \label{eq:dh4}\\ &-2\int hv_{jt}\Big[ {(\phib+v_j+z)}^p-{(\phib+v_j)}^p -p{(\phib+v_j)}^{p-1}z\Big] \label{eq:dh5}\\ &-\int {\Big[z_{xx}+{(\phib+v_j+z)}^p-{(\phib+v_j)}^p\Big]}^2h_x -2\int z_{xx}^2h_x. \label{eq:dh6} \end{align} To conclude, we estimate each term of this equality:
\begin{itemize}
\item First note that $\eqref{eq:dh6}\geq 0$ since $h_x<0$.
\item \eqref{eq:dh1}: By the expression of $h$ and $|\psi'''|\leq \frac{\sigma_0}{4}|\psi'|$, we see after direct calculation that $h_t\geq \sigma_0|h_x|\geq 4|h_{xxx}|$, thus \[ \eqref{eq:dh1} \geq \frac{3}{4}\int z_x^2 h_t -\int F(z)h_t \geq -\int |F(z)|h_t. \] Moreover, since $\nli{Rh_t}\leq Ce^{-4\c t}$, and \begin{align*} |F(z)| &\leq C{|z|}^{p+1} +Cz^2{|\phib+v_j|}^{p-1} \leq C\nli{z}^{p-1}z^2+Cz^2({|\phib|}^{p-1}+{|v_j|}^{p-1})\\ &\leq C\nli{z}z^2 +Cz^2{|\phib-R|}^{p-1} +Cz^2{|R|}^{p-1} +Cz^2\nli{v_j}, \end{align*} then $\int |F(z)|h_t \leq C_0\nh{z}^3 +C_1e^{-2\c t}\nh{z}^2$.
\item For \eqref{eq:dh5}, first note that $\nli{v_{jt}}\leq Ce^{-e_jt}$, and so \[ |\eqref{eq:dh5}|\leq C\nli{v_{jt}}\nld{z}^2 \leq C_1e^{-2\c t}\nh{z}^2. \]
\item $|\eqref{eq:dh4}|\leq C{\|\Omega\|}_{H^3} \nld{z} \leq C_2e^{-(e_j+4\c)t}\nh{z}$ by \eqref{Omega}.
\item To estimate \eqref{eq:dh2}, we develop it as \begin{align*} \frac{1}{2}\eqref{eq:dh2} &= \int z_xh_x\sum_{k=1}^p \binom{p}{k}{\left[{(\phib+v_j)}^{p-k}z^k\right]}_x = \sum_{k=1}^{p-1} \binom{p}{k}k\int z_x^2 z^{k-1} {(\phib+v_j)}^{p-k}h_x\\ &+ \sum_{k=1}^{p-1} \binom{p}{k} (p-k)\int {(\phib+v_j)}_x{(\phib+v_j)}^{p-k-1}h_xz_xz^k + p\int z_x^2z^{p-1}h_x. \end{align*} Since $|\phib_xh_x|+|\phib h_x|\leq Ce^{-2\c t}$ and $|v_{jx}|+|v_j|\leq Ce^{-e_jt}$, then \[ |\eqref{eq:dh2}|\leq C_1e^{-2\c t}\nh{z}^2 + C_0\nh{z}^3. \]
\item We finally estimate \eqref{eq:dh3} to conclude. The key point to control it is that locally around $x=c_kt+x_k$, $\phib$ behaves as a solitary wave of speed $c_k$. More precisely, we strongly use the estimate $\nli{\phib_t h+\phib_x}\leq Ce^{-2\c t}$, proved in \cite{martel:Nsoliton}. Note that the proof uses the $H^4$ norm of the difference $\phib-R$, \emph{i.e.} \eqref{eq:cmm}. Now, we compute \begin{multline*} -\frac{1}{2}\eqref{eq:dh3} = \int z{\Big[ {(\phib+v_j+z)}^p -{(\phib+v_j)}^p -p{(\phib+v_j)}^{p-1}z\Big]}_x\\ +\int \phib_th\Big[ {(\phib+v_j+z)}^p-{(\phib+v_j)}^p -p{(\phib+v_j)}^{p-1}z- \frac{p(p-1)}{2}{(\phib+v_j)}^{p-2}z^2 \Big]\\ -p\int {(\phib+v_j)}^{p-1} z_xz +\frac{p(p-1)}{2}\int \phib_th{(\phib+v_j)}^{p-2}z^2 = \mathbf{I}+ \mathbf{II}+ \mathbf{III}+ \mathbf{IV}. \end{multline*} First notice that $|\mathbf{I}|+|\mathbf{II}|\leq C_0\nh{z}^3$. Moreover, an integration by parts gives \begin{align*} \mathbf{III}+ \mathbf{IV} &= \frac{p}{2}\int z^2(p-1)(\phib_x+v_{jx}){(\phib+v_j)}^{p-2} +\frac{p(p-1)}{2}\int \phib_th{(\phib+v_j)}^{p-2}z^2\\ &= \frac{p(p-1)}{2}\int z^2{(\phib+v_j)}^{p-2}(\phib_x+\phib_th) +\frac{p(p-1)}{2}\int z^2v_{jx}{(\phib+v_j)}^{p-2}, \end{align*} thus \[ |\mathbf{III}+\mathbf{IV}| \leq C\nli{\phib_x+\phib_th}\nld{z}^2 +C\nli{v_{jx}}\nld{z}^2 \leq Ce^{-2\c t}\nh{z}^2 + Ce^{-e_jt}\nh{z}^2, \] and so finally $|\eqref{eq:dh3}|\leq C_0\nh{z}^3 +C_1e^{-2\c t}\nh{z}^2$. \qedhere
\end{itemize}
\end{proof}

We can now prove that, for all $t\in [\ta,S_n]$, \begin{equation} \label{eq:Lh} \int \Big(z_x^2(t)-pR^{p-1}(t) z^2(t)\Big)h(t)+z^2(t) \leq K_1e^{-2(e_j+2\c)t}. \end{equation} Indeed, from Lemma \ref{th:dHdt} and estimates \eqref{eq:ta}, we deduce that, for all $t\in [\ta,S_n]$, \[ \frac{dH}{dt}(t) \geq -C_0e^{-3(e_j+\c)t} -C_1e^{-2\c t}e^{-2(e_j+\c)t} -C_2e^{-(e_j+3\c)t}e^{-(e_j+\c)t} \geq -K_1e^{-2(e_j+2\c)t}. \] Thus by integration on $[t,S_n]$, we obtain $H(S_n)-H(t)\geq -K_1e^{-2(e_j+2\c)t}$, and so \[ H(t)\leq H(S_n)+K_1e^{-2(e_j+2\c)t}. \] But from Claim \ref{Sn} and Lemma \ref{th:finaldata}, we have \begin{align*} H(S_n) &\leq |H(S_n)|\leq C\nh{z(S_n)}^2 \leq C{\|\b\|}^2 \leq C{\|\a\|}^2\\ &\leq Ce^{-2(e_j+2\c)S_n} \leq Ce^{-2(e_j+2\c)t}, \end{align*} and so \[ \forall t\in [\ta,S_n],\quad H(t)\leq K_1e^{-2(e_j+2\c)t}. \] Finally, since \begin{align*} |F(z)-pR^{p-1}z^2| &\leq |F(z)-p{(\phib+v_j)}^{p-1}z^2| +p|({(\phib+v_j)}^{p-1}-\phib^{p-1})z^2| +p|(\phib^{p-1}-R^{p-1})z^2|\\ &\leq C_0{|z|}^3 +C_1e^{-2\c t}{|z|}^2, \end{align*} we easily obtain \eqref{eq:Lh} from the definition of $H$.

\subsubsection{Control of the $R_{kx}$ directions} \label{sec:controlRkx}

Define $\displaystyle \zt(t) = z(t)+\sum_{k=1}^N a_k(t)R_{kx}(t)$, where $a_k(t) = -\frac{\int z(t)R_{kx}(t)}{\int (Q'_{c_k})^2}$, so that by \eqref{eq:interact} \begin{equation} \label{eq:Rkx} \left| \int \zt R_{kx} \right| \leq Ce^{-\c t}\nh{z} \end{equation} and there exist $C_1,C_2>0$ such that \begin{equation} \label{eq:zzt} C_1\nh{z}\leq \nh{\zt}+\sum_{k=1}^N |a_k| \leq C_2\nh{z}. \end{equation} As in \cite[Section 4]{martel:Nsoliton}, we find \[ \int \left[ (\zt_x^2 -pR^{p-1}\zt^2)h+\zt^2\right] \leq \int \left[ (z_x^2-pR^{p-1}z^2)h+z^2\right] +Ce^{-2\c t}\nh{z}^2. \] Using \eqref{eq:Lh}, we deduce that \begin{equation} \label{eq:Lht} \forall t\in [\ta,S_n],\quad  \int \Big(\zt_x^2(t) -pR^{p-1}(t)\zt^2(t)\Big)h(t)+\zt^2(t) \leq K_1e^{-2(e_j+2\c)t}. \end{equation}

Now, from the property of coercivity (vi) in Lemma \ref{th:Zc}, and since $h$ takes values close to $\frac{1}{c_k}$ for $x$ close to $c_kt+x_k$, we obtain, by simple localization arguments (see \cite[Lemma 4]{martel:tsai} for details), that there exists $\lambda_2>0$ such that \[ \int (\zt_x^2-pR^{p-1}\zt^2)h+\zt^2 \geq \lambda_2\nh{\zt}^2 -\frac{1}{\lambda_2}\sum_{k=1}^N \left[ {\left( \int \zt R_{kx} \right)}^2 + {\left( \int \zt Z_k^+ \right)}^2 + {\left( \int \zt Z_k^- \right)}^2 \right]. \] Moreover, gathering all previous estimates, we have for all $t\in [\ta,S_n]$:
\begin{enumerate}[(a)]
\item For all $k\in\unn$, ${\left( \int \zt R_{kx}\right)}^2 \leq Ce^{-2\c t}\nh{z}^2 \leq Ce^{-2(e_j+2\c)t}$ by \eqref{eq:Rkx}.
\item For all $k\in\unn$, ${\left( \int \zt Z_k^+\right)}^2 \leq 2{(\alpha_k^+)}^2+ Ce^{-2\c t}\nh{z}^2 \leq Ce^{-2(e_j+2\c)t}$ by (iii) of Lemma \ref{th:Zc}, \eqref{eq:alphap} and \eqref{eq:interact}.
\item For all $k\in \ient{1}{j}$, ${\left( \int \zt Z_k^-\right)}^2 \leq 2{(\alpha_k^-)}^2+ Ce^{-2\c t}\nh{z}^2 \leq Ce^{-2(e_j+2\c)t}$ by (iii) of Lemma \ref{th:Zc}, \eqref{eq:alphambis} and \eqref{eq:interact}.
\item For all $k\in \ient{j+1}{N}$, ${\left( \int \zt Z_k^-\right)}^2 \leq 2{(\alpha_k^-)}^2+ Ce^{-2\c t}\nh{z}^2 \leq Ce^{-2(e_j+2\c)t}$ by \eqref{eq:ta}.
\end{enumerate}
Finally, we have proved that there exists $K>0$ such that, for all $t\in [\ta,S_n]$, \[ \nh{\zt(t)}\leq Ke^{-(e_j+2\c)t}. \] We want now to prove the same estimate for $z$.

\begin{lem} \label{th:zdecroit}
There exists $K_0>0$ such that, for all $t\in [\ta,S_n]$, \[ \nh{z(t)}\leq K_0e^{-(e_j+2\c)t}. \]
\end{lem}

\begin{proof}
By \eqref{eq:zzt}, it is enough to prove this estimate for $|a_k(t)|$ with $k\in\unn$ fixed. To do this, write first the equation of $\zt$ from the equation of $z$ \eqref{eq:z2}:
\begin{align*} &\zt_t +{(\zt_{xx}+p\phib^{p-1}\zt)}_x\\ &= z_t +\sum_{l=1}^N a_lR_{lxt} +\sum_{l=1}^N a'_lR_{lx} +z_{xxx} +\sum_{l=1}^N a_lR_{lxxx} +p\sum_{l=1}^N a_l{(R_{lx}\phib^{p-1})}_x +p{(\phib^{p-1}z)}_x\\ &= -{(\omega_1\cdot z)}_x -{(\omega(z))}_x -\Omega_x +\sum_{l=1}^N a'_lR_{lx} +\sum_{l=1}^N a_l{\Big[ -c_lR_{lx}+R_{lxxx} +p\phib^{p-1}R_{lx} \Big]}_x. \end{align*} Then multiply this equation by $R_{kx}$ and integrate, so that we obtain \begin{multline*} \int \zt_tR_{kx} -\int (\zt_{xx}+p\phib^{p-1}\zt)R_{kxx} = a'_k\int R_{kx}^2 +\sum_{l\neq k} a'_l \int R_{lx}R_{kx}\\ +\sum_{l=1}^N a_l\int {\Big[ R_{lxxx}-c_lR_{lx}+p\phib^{p-1}R_{lx}\Big]}_x R_{kx} +\int \omega_1zR_{kxx} +\int \omega(z)R_{kxx} +\int \Omega R_{kxx}. \end{multline*}

But from \eqref{eq:cmm} and (iii) of Claim \ref{th:solitons}, we have \begin{align*} \nli{{(R_{lxxx}-c_lR_{lx}+p\phib^{p-1}R_{lx})}_x} &\leq p{\| R_{lx}(\phib^{p-1}-R_l^{p-1})\|}_{H^2}\\ &\leq C{\|\phib-R\|}_{H^2} +p{\|R_{lx}(R^{p-1}-R_l^{p-1})\|}_{H^2} \leq Ce^{-2\c t}, \end{align*} and consequently \begin{align*} |a'_k| &\leq C\left| \int \zt_tR_{kx}\right| +C\nld{\zt} +Ce^{-\c t}\sum_{l\neq k} |a'_l| +Ce^{-2\c t}\sum_{l=1}^N |a_l|\\ &\qquad +Ce^{-e_jt}\nld{z} +C\nld{z}^2 +C\nld{\Omega}. \end{align*} Moreover, from $\int \zt R_{kx} = \sum_{l\neq k} a_l\int R_{lx}R_{kx}$, we deduce that \begin{align*} \dt\int \zt R_{kx} &= \sum_{l\neq k} a'_l \int R_{kx}R_{lx} +\sum_{l\neq k} a_l\int (-c_lR_{lxx}R_{kx}-c_kR_{lx}R_{kxx})\\ &= \int \zt_t R_{kx} +\int \zt(-c_kR_{kxx}), \end{align*} and so \[  \left| \int \zt_t R_{kx} \right| \leq C\nh{\zt} +Ce^{-\c t}\sum_{l\neq k} |a'_l| +Ce^{-2\c t}\sum_{l=1}^N |a_l|. \]

Gathering previous estimates, we have from \eqref{eq:zzt} and \eqref{Omega}, \begin{align*} |a'_k| &\leq C\nh{\zt} + C_4e^{-\c t}\sum_{l\neq k} |a'_l| +Ce^{-2\c t}\nh{z}+C\nh{z}^2 +C\nld{\Omega}\\ &\leq Ke^{-(e_j+2\c)t} + C_4e^{-\c t}\sum_{l\neq k} |a'_l|+ Ce^{-2\c t}e^{-(e_j+\c)t} +Ce^{-2(e_j+\c)t} +Ce^{-(e_j+4\c)t}. \end{align*} Finally, if we choose $t_0$ large enough so that $C_4e^{-\c t_0}\leq \frac{1}{N}$, we obtain for all $s\in [\ta,S_n]$, \[ |a'_k(s)|\leq Ke^{-(e_j+2\c)s}. \] By integration on $[t,S_n]$ with $t\in [\ta,S_n]$, we get $|a_k(t)|\leq |a_k(S_n)| +Ke^{-(e_j+2\c)t}$. But from Claim \ref{Sn} and Lemma \ref{th:finaldata}, we have \[ |a_k(S_n)| \leq C\nh{z(S_n)} \leq C\|\b\| \leq C\|\a\| \leq Ce^{-(e_j+2\c)S_n} \leq Ce^{-(e_j+2\c)t}, \] and so finally, \[ \forall t\in [\ta,S_n],\quad |a_k(t)|\leq Ke^{-(e_j+2\c)t}. \qedhere \]
\end{proof}

\subsubsection{Control of the unstable directions for $k>j$ by a topological argument}

Lemma \ref{th:zdecroit} being proved, we choose $t_0$ large enough so that $K_0e^{-\c t_0}\leq \frac{1}{2}$. Therefore, we have \[ \forall t\in [\ta,S_n],\quad \nh{z(t)}\leq \frac{1}{2}e^{-(e_j+\c)t}. \] We can now prove the following final lemma, which concludes the proof of Proposition \ref{th:princbis}.

\begin{lem}
For $t_0$ large enough, there exists $\a\in B_{\R^{N-j}}(e^{-(e_j+2\c)S_n})$ such that $\ta=t_0$.
\end{lem}

\begin{proof}
For the sake of contradiction, suppose that for all $\a\in B_{\R^{N-j}}(e^{-(e_j+2\c)S_n})$, $\ta>t_0$. As $e^{(e_j+\c)\ta}z(\ta)\in B_{H^1}(1/2)$, then by definition of $\ta$ and continuity of the flow, we have \begin{equation} \label{eq:alphamta} e^{(e_j+2\c)\ta}\alpham(\ta) \in\S_{\R^{N-j}}(1). \end{equation} Now let $T\in [t_0,\ta]$ be close enough to $\ta$ such that $z$ is defined on $[T,S_n]$, and by continuity, \[ \forall t\in [T,S_n],\quad \nh{z(t)}\leq e^{-(e_j+\c)t}. \] We can now consider, for $t\in [T,S_n]$, \[ \N(t) = \N(\alpham(t)) = \| e^{(e_j+2\c)t}\alpham(t)\|^2. \] To calculate $\N'$, we start from estimate \eqref{eq:alpham}: \[ \forall k\in \ient{j+1}{N}, \forall t\in [T,S_n],\quad \left| \dt \alpha_k^-(t) +e_k\alpha_k^-(t)\right| \leq K'_2e^{-(e_j+4\c)t}. \] Multiplying by $|\alpha_k^-(t)|$, we obtain \[ \left| \alpha_k^-(t)\dt\alpha_k^-(t) +e_k{\alpha_k^-(t)}^2 \right| \leq K'_2e^{-(e_j+4\c)t}|\alpha_k^-(t)|, \] and thus \[ 2\alpha_k^-(t)\dt\alpha_k^-(t) +2e_{j+1}{\alpha_k^-(t)}^2 \leq 2\alpha_k^-(t)\dt\alpha_k^-(t) +2e_k{\alpha_k^-(t)}^2 \leq K_2e^{-(e_j+4\c)t}|\alpha_k^-(t)|. \] By summing on $k\in \ient{j+1}{N}$, we get \[ {(\|\alpham(t)\|^2)}' +2e_{j+1}\|\alpham(t)\|^2 \leq K_2e^{-(e_j+4\c)t}\|\alpham(t)\|. \]

Therefore we can estimate \begin{align*} \N'(t) &= {(e^{2(e_j+2\c)t}\|\alpham(t)\|^2)}' = e^{2(e_j+2\c)t}\left[ 2(e_j+2\c)\|\alpham(t)\|^2 + {(\|\alpham(t)\|^2)}' \right]\\ &\leq e^{2(e_j+2\c)t} \left[ 2(e_j+2\c)\|\alpham(t)\|^2 -2e_{j+1}\|\alpham(t)\|^2 +K_2e^{-(e_j+4\c)t}\|\alpham(t)\| \right]. \end{align*} Hence we have, for all $t\in [T,S_n]$, \[ \N'(t) \leq -\theta\cdot\N(t) +K_2e^{e_jt}\|\alpham(t)\|, \] where $\theta=2(e_{j+1}-e_j-2\c)>0$ by definition of $\c$ \eqref{gamma}. In particular, for all $\tau\in [T,S_n]$ satisfying $\N(\tau)=1$, we have \[ \N'(\tau) \leq -\theta +K_2e^{e_j\tau}\|\alpham(\tau)\| = -\theta +K_2e^{e_j\tau}e^{-(e_j+2\c)\tau} = -\theta +K_2e^{-2\c\tau} \leq -\theta +K_2e^{-2\c t_0}. \] Now we fix $t_0$ large enough so that $K_2e^{-2\c t_0}\leq \frac{\theta}{2}$, and so for all $\tau\in [T,S_n]$ such that $\N(\tau)=1$, we have \begin{equation} \label{eq:Nprime} \N'(\tau) \leq -\frac{\theta}{2}. \end{equation} In particular, by \eqref{eq:alphamta}, we have $\N'(\ta) \leq -\frac{\theta}{2}$.

\begin{description}
\item[First consequence:] $\a\mapsto \ta$ is continuous. Indeed, let $\eps>0$. Then there exists $\delta>0$ such that $\N(\ta-\eps)>1+\delta$ and $\N(\ta+\eps)<1-\delta$. Moreover, by definition of $\ta$ and \eqref{eq:Nprime}, there can not exist $\tau\in [\ta+\eps,S_n]$ such that $\N(\tau)=1$, and so by choosing $\delta$ small enough, we have for all $t\in [\ta+\eps,S_n]$, $\N(t)<1-\delta$. But from continuity of the flow, there exists $\eta>0$ such that, for all $\widetilde{\mathfrak{a}}^-$ satisfying $\|\widetilde{\mathfrak{a}}^- -\a\|\leq \eta$, we have \[ \forall t\in [\ta-\eps,S_n],\quad |\N(\widetilde{\boldsymbol{\alpha}}^-(t)) -\N(\alpham(t))| \leq \delta/2. \] We finally deduce that $\ta-\eps \leq T(\widetilde{\mathfrak{a}}^-)\leq \ta+\eps$, as expected.
\item[Second consequence:] We can define the map \[ \begin{array}{rrcl} \M~: &B_{\R^{N-j}}(e^{-(e_j+2\c)S_n}) &\to &\S_{\R^{N-j}}(e^{-(e_j+2\c)S_n})\\ &\a &\mapsto &e^{-(e_j+2\c)(S_n-\ta)}\alpham(\ta). \end{array} \] Note that $\M$ is continuous by the previous point. Moreover, let $\a\in\S_{\R^{N-j}}(e^{-(e_j+2\c)S_n})$. As $\N'(S_n)\leq -\frac{\theta}{2}$ by \eqref{eq:Nprime}, we deduce by definition of $\ta$ that $\ta=S_n$, and so $\M(\a)=\a$. In other words, $\M$ restricted to $\S_{\R^{N-j}}(e^{-(e_j+2\c)S_n})$ is the identity. But the existence of such a map $\M$ contradicts Brouwer's fixed point theorem.
\end{description}

In conclusion, there exists $\a\in B_{\R^{N-j}}(e^{-(e_j+2\c)S_n})$ such that $\ta=t_0$.
\end{proof}

\section{Classification of multi-solitons} \label{sec:classification}

This section is devoted to prove the second assertion of Theorem \ref{th:main}. Let $p>5$, $N\geq 2$, $0<c_1<\cdots<c_N$ and $x_1,\ldots,x_N\in\R$. Denote $R=\sum_{j=1}^N R_{c_j,x_j}$ and $\phib$ the multi-soliton given by Theorem \ref{th:cmm}. Let $u$ be a solution of \eqref{eq:gKdV}, defined on $[t_1,+\infty)$ with $t_1>0$ large, satisfying \begin{equation} \label{eq:hypo} \lim_{t\to+\infty} \nh{u(t)-R(t)} = 0. \end{equation}

\subsection{Convergence at exponential rate $\c$} \label{sec:expoc}

We first improve condition \eqref{eq:hypo} into an exponential convergence, with a small rate $\c>0$, where $\c$ is defined by \eqref{gamma}.

\begin{lem}
Let $\eps=u-\phib$. Then there exist $C,t_0>0$ such that, for all $t\geq t_0$, $\nh{\eps(t)}\leq Ce^{-\c t}$.
\end{lem}

\begin{proof}
\textit{Step 1: Modulation.} Denote $v=u-R$, so that $\nh{v(t)}\to 0$ as $t\to+\infty$ by \eqref{eq:hypo}. Therefore, by a standard lemma of modulation (see for example \cite[Lemma 2]{martel:Nsoliton}), for $t_0$ large enough, there exist $N$ functions $y_j~: [t_0,+\infty) \to\R$ of class $C^1$ such that $w=u-\Rt$, where $\Rt = \sum \Rt_j$ and $\Rt_j(t) = R_j(t,\cdot-y_j(t))$, satisfies \[ \begin{cases} \forall j\in\unn,\quad \int w(t){(\Rt_j)}_x(t)=0,\\ \nh{w(t)}+\sum_{j=1}^N |y_j(t)| \leq C\nh{v(t)},\\ \forall j\in\unn,\quad |y'_j(t)|\leq C\nh{w(t)}+Ce^{-\c t}. \end{cases} \] Note that the first two facts are a simple consequence of the implicit function theorem, while the last estimate comes from the equation satisfied by $w$, \[ \partial_t w + \partial_x^3 w = \sum_{k=1}^N y'_k \partial_x(\Rt_k) - \partial_x \left( {(w+\Rt)}^p -\sum_{k=1}^N {\Rt_k}^p \right), \] multiplied by ${(\Rt_j)}_x$ and integrated. Similarly, if we denote $\Zt_j^{\pm}(t)=Z_j^{\pm}(t,\cdot-y_j(t))$ and $\alphat_j^{\pm}(t) = \int w(t)\Zt_j^{\pm}(t)$, the equation of $w$ multiplied by $\Zt_j^{\pm}$ leads to \begin{equation} \label{eq:alphatilde} \forall t\geq t_0,\quad \left| \dt \alphat_j^{\pm}(t)\mp e_j \alphat_j^{\pm}(t)\right| \leq C\nh{w(t)}^2 +Ce^{-2\c t}. \end{equation}

\textit{Step 2: Monotonicity.} We use again the function $\psi$ introduced in Section \ref{sec:monotonicity}. Following \cite{martel:Nsoliton}, we introduce moreover $\psi_N \equiv 1$ and for $j\in\ient{1}{N-1}$, \[ m_j(t) = \frac{c_j+c_{j+1}}{2}t +\frac{x_j+x_{j+1}}{2},\quad \psi_j(t)=\psi(x-m_j(t)), \] and \[ \phi_1\equiv \psi_1,\quad \phi_N\equiv 1-\psi_{N-1},\quad \phi_j \equiv \psi_j-\psi_{j-1}\quad \m{for}\ j\in\ient{2}{N-1}. \] We also define some local quantities related to $L^2$ mass and energy: \[ M_j(t) = \int u^2(t)\phi_j(t),\quad E_j(t) = \int \left(\frac{1}{2}u_x^2(t) -\frac{1}{p+1}u^{p+1}(t)\right)\phi_j(t),\quad \wt E_j(t) = E_j(t) + \frac{\sigma_0}{100}M_j(t). \] Then, by \eqref{eq:hypo} and monotonicity results on the quantities $t\mapsto \sum_{k=1}^j M_k(t)$ and $t\mapsto \sum_{k=1}^j E_k(t)$, we have, for all $t\geq t_0$ and all $j\in\unn$, following Lemmas 1 and 3 of \cite{martel:Nsoliton}, \begin{empheq}[left=\empheqlbrace\,]{align} &\sum_{k=1}^j \left( \int Q_{c_k}^2 -M_k(t)\right) \geq -K_2e^{-2\c t}, \label{eq:locmass} \\ &\sum_{k=1}^j \left( E(Q_{c_k})+\frac{\sigma_0}{100}\int Q_{c_k}^2 -\wt E_k(t)\right) \geq -K_2e^{-2\c t}, \label{eq:locenergy} \end{empheq} and \begin{equation} \label{eq:locweinstein} \left| \left( E_j(t)+\frac{c_j}{2}M_j(t)\right) -\left( E(Q_{c_j})+\frac{c_j}{2}\int Q_{c_j}^2\right) -\frac{1}{2}H_j(t)\right| \leq K_4e^{-2\c t} +K_4\nh{w(t)}\int w^2\phi_j, \end{equation} where $H_j(t) = \int \left( w_x^2(t)+c_jw^2(t)-p\Rt_j^{p-1}(t)w^2(t)\right)\phi_j(t)$. But if we write \begin{multline*} \sum_{j=1}^N \frac{1}{c_j^2}\left(E_j+\frac{c_j}{2}M_j\right) = \sum_{j=1}^{N-1} \left[ \left( \frac{1}{c_j^2}-\frac{1}{c_{j+1}^2}\right) \sum_{k=1}^j \wt E_k\right] +\frac{1}{c_N^2}\sum_{k=1}^N \wt E_k +\frac{1}{2c_N}\left(1-\frac{\sigma_0}{50c_N}\right) \sum_{k=1}^N M_k\\ +\sum_{j=1}^{N-1} \left[ \frac{1}{2}\left( \frac{1}{c_j}-\frac{1}{c_{j+1}}\right)\left( 1-\frac{\sigma_0}{50}\left(\frac{1}{c_j}+\frac{1}{c_{j+1}}\right) \right) \sum_{k=1}^j M_k \right], \end{multline*} and similarly \begin{multline*} \sum_{j=1}^N \frac{1}{c_j^2} \left( E(Q_{c_j})+\frac{c_j}{2}\int Q_{c_j}^2\right) = \sum_{j=1}^{N-1} \left[ \left( \frac{1}{c_j^2}-\frac{1}{c_{j+1}^2}\right) \sum_{k=1}^j \left( E(Q_{c_k})+\frac{\sigma_0}{100}\int Q_{c_k}^2\right) \right]\\ +\sum_{j=1}^{N-1} \left[ \frac{1}{2}\left( \frac{1}{c_j}-\frac{1}{c_{j+1}}\right) \left( 1-\frac{\sigma_0}{50}\left( \frac{1}{c_j}+\frac{1}{c_{j+1}}\right) \right)\sum_{k=1}^j \int Q_{c_k}^2 \right]\\ +\frac{1}{c_N^2}\sum_{k=1}^N \left( E(Q_{c_k}) +\frac{\sigma_0}{100}\int Q_{c_k}^2\right) +\frac{1}{2c_N}\left( 1-\frac{\sigma_0}{50c_N}\right)\sum_{k=1}^N \int Q_{c_k}^2, \end{multline*} and we remark that all coefficients in these decompositions are positive, we obtain by \eqref{eq:locmass} and \eqref{eq:locenergy}, \[ \sum_{j=1}^N \frac{1}{c_j^2}\left( E_j(t)+\frac{c_j}{2}M_j(t)\right) -\sum_{j=1}^N \frac{1}{c_j^2}\left( E(Q_{c_j}) +\frac{c_j}{2}\int Q_{c_j}^2\right) \leq Ce^{-2\c t}. \] Therefore, we have by \eqref{eq:locweinstein}, \begin{align*} \frac{1}{2}\sum_{j=1}^N \frac{1}{c_j^2} H_j(t) &\leq \sum_{j=1}^N \frac{1}{c_j^2}\left( E_j(t)+\frac{c_j}{2}M_j(t)\right) -\sum_{j=1}^N \frac{1}{c_j^2} \left( E(Q_{c_j})+\frac{c_j}{2}\int Q_{c_j}^2\right)\\ &\qquad +K_4\sum_{j=1}^N \frac{1}{c_j^2}e^{-2\c t} +K_4\nh{w(t)}\sum_{j=1}^N \frac{1}{c_j^2}\int w^2\phi_j\\ &\leq C_1e^{-2\c t} +\frac{K_4}{\sigma_0^2}\nh{w(t)}\int w^2\sum_{j=1}^N \phi_j \end{align*} since $\phi_j\geq 0$. Finally, as $\sum_{j=1}^N \phi_j \equiv 1$, we obtain \begin{equation} \label{eq:Hj} \sum_{j=1}^N \frac{1}{c_j^2} H_j(t)\leq C_1e^{-2\c t} +C_2\nh{w(t)}^3. \end{equation}

\textit{Step 3: Coercivity.} Now, from the property of coercivity (vi) in Lemma \ref{th:Zc} and by standard localization arguments (as in Section \ref{sec:construction}), we have  \[ \sum_{j=1}^N \frac{1}{c_j^2} H_j(t) \geq \lambda_c\nh{w(t)}^2 -\frac{1}{\lambda_c}\sum_{j=1}^N {\left( \int w(t){(\Rt_j)}_x(t)\right)}^2 -\frac{1}{\lambda_c} \sum_{j,\pm} {\left( \int w(t)\Zt_j^{\pm}(t)\right)}^2. \] As $\int w(t){(\Rt_j)}_x(t)=0$ and $\alphat_j^{\pm}(t) = \int w(t)\Zt_j^{\pm}(t)$, we obtain by \eqref{eq:Hj}, \[ \lambda_c\nh{w(t)}^2 \leq C_1e^{-2\c t}+C_2\nh{w(t)}^3 +C_3\|\alphatv(t)\|^2, \] where $\alphatv(t) = {(\alphat_j^{\pm}(t))}_{j,\pm}$. For $t_0$ large enough so that $C_2\nh{w(t)}\leq \frac{\lambda_c}{2}$, we obtain \begin{equation} \label{eq:walpha} \forall t\geq t_0,\quad \nh{w(t)}^2 \leq C_1\|\alphatv(t)\|^2 +C_2e^{-2\c t}. \end{equation}

\textit{Step 4: Exponential decay of $\alphatv$.} From \eqref{eq:alphatilde} and \eqref{eq:walpha}, we have for all $j\in\unn$ and all $t\geq t_0$, \[ \left|\dt \alphat_j^{\pm}(t) \mp e_j\alphat_j^{\pm}(t)\right| \leq C_1\|\alphatv(t)\|^2 +C_2e^{-2\c t}. \] We follow here the strategy of \cite[Section 4.4.2]{combet}. Define $A(t)=\sum_{j=1}^N {\alphat_j^+(t)}^2$ and $B(t) = \sum_{j=1}^N {\alphat_j^-(t)}^2$, and let us prove that $A(t)\leq B(t)+Le^{-2\c t}$ for $L$ large enough. First, we have, by multiplying the previous estimate by $|\alphat_j^+(t)|$ (that we can of course suppose less than $1$), \[ \alphat_j^+(t)\dt\alphat_j^+(t) \geq e_j{\alphat_j^+(t)}^2 -C_1 |\alphat_j^+(t)|\cdot \|\alphatv(t)\|^2 -C_2e^{-2\c t}, \] and so by summing, \[ A'(t) \geq 2e_1A(t)-C_1 \|\alphatv(t)\|^3 -C_2e^{-2\c t}. \] Similarly, we obtain \begin{equation} \label{eq:Bprime} B'(t) \leq -2e_1B(t) +C_1 \|\alphatv(t)\|^3 +C_2e^{-2\c t}. \end{equation} Now let $h(t) = A(t)-B(t)-Le^{-2\c t}$ with $L$ to be determined later. We have of course $h(t)\to 0$ as $t\to+\infty$, and by the previous estimates, we can calculate \begin{align*} h'(t) &= A'(t)-B'(t)+2L\c e^{-2\c t}\\ &\geq 2e_1A(t)+2e_1B(t) -C_1\|\alphatv(t)\|^3 -C_2e^{-2\c t}\\ &\geq 2e_1h(t) +4e_1B(t) -C_1\|\alphatv(t)\|^3 -C_2e^{-2\c t} +2Le_1 e^{-2\c t}. \end{align*} Since $\|\alphatv(t)\|^2 = A(t)+B(t) = h(t)+2B(t)+Le^{-2\c t}$, we get \[ h'(t)\geq h(t)(2e_1-C_1\|\alphatv(t)\|) +B(t)(4e_1-2C_1\|\alphatv(t)\|) +e^{-2\c t}(2Le_1 -C_2 -C_1L\|\alphatv(t)\|). \] Now choose $t_0$ large enough so that $C_1\|\alphatv(t)\|\leq \frac{e_1}{2}$ for $t\geq t_0$, and fix $L = \frac{C_2}{e_1}$. Therefore, we have, for all $t\geq t_0$ such that $h(t)\geq 0$, $h'(t)\geq e_1h(t)$. Hence, if there exists $T\geq t_0$ such that $h(T)\geq 0$, then $h(t)\geq 0$ for all $t\geq T$, and thus $h(t)\geq Ce^{e_1t}$, which would be in contradiction with $\lim_{t\to+\infty} h(t)=0$. So we have proved that $h(t)\leq 0$ for all $t\geq t_0$, as expected.

Now, from \eqref{eq:Bprime} and the choice of $t_0$ to have $C_1\|\alphatv(t)\|\leq\frac{e_1}{2}$ for all $t\geq t_0$, it comes \[ B'(t)+2e_1B(t) \leq e_1B(t) +\left( \frac{Le_1}{2}+C_2\right)e^{-2\c t}, \] and so $B'(t)+e_1B(t)\leq Ke^{-2\c t}$. Therefore, ${(e^{e_1s}B(s))}' \leq Ke^{(e_1-2\c)s}$ for $s\geq t_0$, and so by integration on $[t_0,t]$, \[ e^{e_1t}B(t)-e^{e_1t_0}B(t_0) \leq Ke^{(e_1-2\c)t}, \] since $e_1-2\c>0$. We deduce that \[ B(t) \leq Ke^{-2\c t}+K'e^{-e_1t} \leq Ke^{-2\c t}. \] Finally, we also have by the previous point $A(t)\leq K'e^{-2\c t}$, and so \begin{equation} \label{eq:alphadecroit} \forall t\geq t_0,\quad \|\alphatv(t)\|^2 \leq Ce^{-2\c t}. \end{equation}

\textit{Step 5: Conclusion.} By \eqref{eq:walpha}, we deduce that $\nh{w(t)}\leq Ce^{-\c t}$, and from the estimate on $|y'_j|$, we have for all $j\in\unn$ and all $t\geq t_0$, $|y_j(t)|\leq Ce^{-\c t}$, by integration and the fact that $y_j(t)\to 0$ as $t\to+\infty$. To conclude, write \[ \eps = u-\phib = w+\Rt-\phib = w -(\phib-R) +(\Rt-R), \] so that \[ \nh{\eps(t)} \leq \nh{w(t)}+\nh{(\phib-R)(t)} +{\| (\Rt-R)(t)\|}_{H^1} \leq Ce^{-\c t}+{\|(\Rt-R)(t)\|}_{H^1}. \] But we have \begin{align*} {\|(\Rt-R)(t)\|}_{H^1} &\leq \sum_{j=1}^N \nh{R_j(t,\cdot-y_j(t))-R_j(t)}\\ &\leq C\sum_{j=1}^N \nh{Q(\sqrt{c_j}x-\sqrt{c_j}y_j(t)-c_j^{3/2}t-\sqrt{c_j}x_j) -Q(\sqrt{c_j}x-c_j^{3/2}t-\sqrt{c_j}x_j)}\\ &\leq C\sum_{j=1}^N |y_j(t)|\leq Ce^{-\c t}, \end{align*} and so finally, for all $t\geq t_0$, $\nh{\eps(t)}\leq Ce^{-\c t}$.
\end{proof}

\subsection{Convergence at exponential rate $e_1$}

Now, we improve the convergence of the previous lemma, with an exponential rate $e_1\gg \c$. The proof will mainly use arguments developed in \cite[Section 4]{martel:Nsoliton}.

\begin{lem} \label{th:e1}
There exist $C,t_0>0$ such that, for all $t\geq t_0$, $\nh{\eps(t)}\leq Ce^{-e_1t}$.
\end{lem}

\begin{proof}
\textit{Step 1: Estimates.} We follow the same strategy as in Section \ref{sec:preuveprinc}. First, from the equation of $\eps$, \[ \eps_t +{(\eps_{xx} +{(\phib+\eps)}^p -\phib^p)}_x =0, \] we can estimate $\alpha_j^{\pm}(t) = \int \eps(t)Z_j^{\pm}(t)$ for $j\in\unn$ and $t\geq t_0$. Indeed, we have \begin{align*} \dt \alpha_j^{\pm}(t) &= \int \eps_tZ_j^{\pm} +\int \eps Z_{jt}^{\pm} = \int \Big(\eps_{xx}+{(\phib+\eps)}^p-\phib^p\Big) Z_{jx}^{\pm} -c_j\int \eps Z_{jx}^{\pm}\\ &= \int \left[ \eps_{xx}-c_j\eps +\sum_{k=1}^p \binom{p}{k} \phib^{p-k}\eps^k \right] Z_{jx}^{\pm}\\ &= \int \left[ \eps_{xx}-c_j\eps+pR_j^{p-1}\eps\right]Z_{jx}^{\pm} +p\int (\phib^{p-1}-R_j^{p-1})\eps Z_{jx}^{\pm} +\sum_{k=2}^p \binom{p}{k} \int \phib^{p-k}\eps^k Z_{jx}^{\pm}\\ &= \mathbf{I} + \mathbf{II} + \mathbf{III}. \end{align*} But we have $\mathbf{I} = \pm e_j\alpha_j^{\pm}(t)$ (see proof of \eqref{eq:alpha}), $|\mathbf{II}| \leq Ce^{-\c t}\nh{\eps(t)}$ and $|\mathbf{III}|\leq C\nh{\eps(t)}^2$, and so, for all $t\geq t_0$ and all $j\in\unn$, \begin{equation} \label{eq:alpharec} \left| \dt \alpha_j^{\pm}(t)\mp e_j\alpha_j(t)\right| \leq Ce^{-\c t}\nh{\eps(t)}. \end{equation}

To control the $R_{jx}$ directions, we proceed exactly as in Section \ref{sec:controlRkx}. Define $\epst(t) = \eps(t) +\sum_{j=1}^N a_j(t)R_{jx}(t)$, where $a_j(t) = -\frac{\int \eps(t)R_{jx}(t)}{\int (Q'_{c_j})^2}$, so that $|\int \epst(t)R_{jx}(t)|\leq Ce^{-\c t}\nh{\eps(t)}$ and \begin{equation} \label{eq:epsaj} C_1\nh{\eps} \leq \nh{\epst}+\sum_{j=1}^N |a_j| \leq C_2\nh{\eps}. \end{equation} As $\nh{\eps(t)}\leq Ce^{-\c t}$, we have exactly as in \cite{martel:Nsoliton}, for all $t\geq t_0$, by monotonicity arguments, \[ \int \Big[ \eps_x^2(t)-pR^{p-1}(t)\eps^2(t)\Big]h(t)+\eps^2(t) \leq Ce^{-2\c t}\sup_{t'\geq t} \nh{\eps(t')}^2, \] where $h$ is defined in Section \ref{sec:monotonicity}. We also have from \cite{martel:Nsoliton}, \[ \int (\epst_x^2-pR^{p-1}\epst^2)h+\epst^2 \leq \int \left[ (\eps_x^2 -pR^{p-1}\eps^2)h+\eps^2\right] +Ce^{-2\c t}\sum_{j=1}^N a_j^2 +Ce^{-2\c t}\nh{\eps}^2, \] and thus \[ \int \Big[ \epst_x^2(t)-pR^{p-1}(t)\epst^2(t)\Big]h(t)+\epst^2(t) \leq Ce^{-2\c t}\sup_{t'\geq t} \nh{\eps(t')}^2. \] But as in Section \ref{sec:monotonicity}, a localization argument of the property of coercivity (vi) in Lemma \ref{th:Zc} leads to \[ \int (\epst_x^2 -pR^{p-1}\epst^2)h +\epst^2 \geq \lambda_2\nh{\epst}^2 -\frac{1}{\lambda_2}\sum_{j=1}^N \left[ {\left( \int \epst R_{jx}\right)}^2 + {\left( \int \epst Z_j^+\right)}^2 + {\left( \int \epst Z_j^-\right)}^2 \right]. \] Since ${\left( \int \epst R_{jx}\right)}^2 \leq Ce^{-2\c t}\nh{\eps(t)}^2$ and ${\left( \int \epst Z_j^{\pm}\right)}^2 \leq 2{(\alpha_j^{\pm})}^2+Ce^{-2\c t}\nh{\eps(t)}^2$, then \[ \lambda_2\nh{\epst}^2 \leq Ce^{-2\c t}\sup_{t'\geq t}\nh{\eps(t')}^2 +Ce^{-2\c t}\nh{\eps(t)}^2 +C\sum_{j=1}^N {(\alpha_j^+)}^2 + C\sum_{j=1}^N {(\alpha_j^-)}^2. \] By denoting $\alphav(t) = {(\alpha_j^{\pm}(t))}_{j,\pm}$, we thus have \begin{equation} \label{eq:epstilde} \nh{\epst(t)}^2 \leq Ce^{-2\c t}\sup_{t'\geq t} \nh{\eps(t')}^2 +C\|\alphav(t)\|^2. \end{equation}

Finally, to estimate $|a_j(t)|$ for all $j\in\unn$, we follow the strategy and some calculation from the proof of Lemma \ref{th:zdecroit}. First write the equation satisfied by $\epst$: \begin{align*} &\epst_t +{(\epst_{xx} +p\phib^{p-1}\epst)}_x\\ &= \eps_t +\eps_{xxx} +p{(\phib^{p-1}\eps)}_x +\sum_{k=1}^N a_kR_{kxt} +\sum_{k=1}^N a'_kR_{kx} +\sum_{k=1}^N a_k R_{kxxx} +p\sum_{k=1}^N a_k{(R_{kx}\phib^{p-1})}_x \\ &= -{\left[{(\phib+\eps)}^p-\phib^p\right]}_x +p{(\phib^{p-1}\eps)}_x +\sum_{k=1}^N a'_kR_{kx} +\sum_{k=1}^N a_k{\left[ -c_kR_{kx} +R_{kxxx} +p\phib^{p-1}R_{kx}\right]}_x\\ &= \sum_{k=1}^N a'_k R_{kx} +\sum_{k=1}^N a_k{\left[ R_{kxxx}-c_kR_{kx}+p\phib^{p-1}R_{kx}\right]}_x -{\left[ {(\phib+\eps)}^p-\phib^p -p\phib^{p-1}\eps\right]}_x. \end{align*} Then multiply by $R_{jx}$ and integrate, so \begin{multline*} \int \epst_tR_{jx} -\int (\epst_{xx}+p\phib^{p-1}\epst)R_{jxx} = a'_j\int R_{jx}^2 +\sum_{k\neq j} a'_k\int R_{kx}R_{jx}\\ +\sum_{k=1}^N a_k\int {\left[ R_{kxxx}-c_kR_{kx}+p\phib^{p-1}R_{kx}\right]}_xR_{jx} +\int \left[ {(\phib^p+\eps)}^p-\phib^p-p\phib^{p-1}\eps\right] R_{jxx}. \end{multline*} As $\nli{{(R_{kxxx}-c_kR_{kx}+p\phib^{p-1}R_{kx})}_x}\leq Ce^{-\c t}$, we obtain \[ |a'_j(t)|\leq C\left|\int \epst_t(t) R_{jx}(t)\right| +Ce^{-\c t}\sum_{k\neq j} |a'_k(t)| +Ce^{-\c t}\nh{\eps(t)} +C\nh{\eps(t)}^2 +C\nh{\epst(t)}. \] Moreover, we still have \[ \left| \int \epst_t(t)R_{jx}(t)\right| \leq C\nh{\epst(t)}+Ce^{-\c t}\sum_{k\neq j} |a'_k(t)| +Ce^{-\c t}\nh{\eps(t)}, \] and so \[ |a'_j(t)|\leq C_1e^{-\c t}\sum_{k\neq j} |a'_k(t)| +Ce^{-\c t}\nh{\eps(t)} +C\nh{\eps(t)}^2 +C\nh{\epst(t)}. \] Finally choose $t_0$ large enough such that $C_1e^{-\c t_0}\leq \frac{1}{N}$, so that we obtain, for all $j\in\unn$ and all $t\geq t_0$, \begin{equation} \label{eq:ajrec} |a'_j(t)|\leq Ce^{-\c t}\nh{\eps(t)} +C\nh{\epst(t)}. \end{equation}

\textit{Step 2: Induction.} With estimates \eqref{eq:alpharec} to \eqref{eq:ajrec}, we can now improve exponential convergence of $\eps$ by a bootstrap argument. We recall that we already have $\nh{\eps(t)}\leq Ce^{-\c_0t}$ with $\c_0=\c$. Now, we prove that if $\nh{\eps(t)}\leq Ce^{-\c_0t}$ with $\c\leq \c_0<e_1-\c$, then $\nh{\eps(t)}\leq C'e^{-(\c_0+\c)t}$. So, suppose that $\nh{\eps(t)}\leq Ce^{-\c_0t}$ with $\c\leq \c_0<e_1-\c$.
\begin{enumerate}[(a)]
\item From \eqref{eq:alpharec}, we get for all $j\in\unn$, $|{(e^{-e_jt}\alpha_j^+(t))}'|\leq Ce^{-(e_j+\c_0+\c)t}$, and so by integration on $[t,+\infty)$, $|\alpha_j^+(t)|\leq Ce^{-(\c_0+\c)t}$, since $\alpha_j^+(t)\to0$ as $t\to+\infty$.
\item Still from \eqref{eq:alpharec}, we get for all $j\in\unn$, $|{(e^{e_jt}\alpha_j^-(t))}'|\leq Ce^{(e_j-\c-\c_0)t}$. As $e_j-\c-\c_0\geq e_1-\c-\c_0>0$, we obtain by integration on $[t_0,t]$, $|e^{e_jt}\alpha_j^-(t) -e^{e_jt_0}\alpha_j^-(t_0)|\leq Ce^{(e_j-\c-\c_0)t}$, and so \[ |\alpha_j^-(t)| \leq Ce^{-(\c_0+\c)t} +Ce^{-e_jt} \leq Ce^{-(\c_0+\c)t}. \]
\item Therefore we have $\|\alphav(t)\|^2\leq Ce^{-2(\c_0+\c)t}$, and so by \eqref{eq:epstilde}, we obtain $\nh{\epst(t)}\leq Ce^{-(\c_0+\c)t}$.
\item From \eqref{eq:ajrec}, we deduce that for all $j\in\unn$, $|a'_j(t)|\leq Ce^{-(\c_0+\c)t}$, and so, by integration on $[t,+\infty)$, $|a_j(t)|\leq Ce^{-(\c_0+\c)t}$, since $a_j(t)\to0$ as $t\to+\infty$.
\item Finally, from \eqref{eq:epsaj}, we have $\nh{\eps(t)}\leq Ce^{-(\c_0+\c)t}$, as expected.
\end{enumerate}

\textit{Step 3: Conclusion.} We apply the previous induction until to have $e_1-\c<\c_0<e_1$. Note that if $\c_0=e_1-\c$, then the estimate is still true for $\c_0=e_1-\frac{3}{2}\c<e_1-\c$, and so for $\c_0=e_1-\frac{1}{2}\c >e_1-\c$ by the previous step. Now we follow the scheme of step 2. We still have, for all $j\in\unn$, $|\alpha_j^+(t)|\leq Ce^{-(\c_0+\c)t}\leq Ce^{-e_1t}$, and $|{(e^{e_jt}\alpha_j^-(t))}'|\leq Ce^{(e_j-\c-\c_0)t}$. In particular, for $j=1$, we have \[ |{(e^{e_1t}\alpha_1^-(t))}'|\leq Ce^{(e_1-\c-\c_0)t} \in L^1([t_0,+\infty)), \] since $e_1-\c-\c_0<0$. Hence there exists $A_1\in\R$ such that \begin{equation} \label{eq:A1} \lim_{t\to+\infty} e^{e_1t}\alpha_1^-(t) = A_1, \end{equation} and $|e^{e_1t}\alpha_1^-(t)-A_1|\leq Ce^{(e_1-\c-\c_0)t}$, and so $|\alpha_1^-(t)|\leq Ce^{-e_1t}$. For $j\geq 2$, since $e_j-\c-\c_0 >e_2-\c-e_1>0$ by definition of $\c$, we still obtain by integration on $[t_0,t]$, $|\alpha_j^-(t)|\leq Ce^{-(\c_0+\c)t}\leq Ce^{-e_1t}$. As in step 2, it follows $\|\alphav(t)\|^2\leq Ce^{-2e_1t}$, then $\nh{\epst(t)}\leq Ce^{-e_1t}$ by \eqref{eq:epstilde}, $|a_j(t)|\leq Ce^{-e_1t}$ for all $j\in\unn$ by \eqref{eq:ajrec}, and finally $\nh{\eps(t)}\leq Ce^{-e_1t}$ by \eqref{eq:epsaj}, as expected.
\end{proof}

\subsection{Identification of the solution}

We now prove the following proposition by induction, following the strategy of the previous section. We identify $u$ among the family $(\phib_{\an})$ constructed in Section \ref{sec:construction}. We recall that this family was constructed thanks to the subfamilies $(\phib_{\aj})$, satisfying \eqref{eq:phiaj} for all $j\in\unn$: \[ \forall t\geq t_0,\quad \nh{\phib_{\aj}(t)-\phib_{A_1,\ldots,A_{j-1}}(t)-A_je^{-e_jt}Y_j^+(t)} \leq e^{-(e_j+\c)t}. \]

\begin{prop} \label{th:rec}
For all $j\in\unn$, there exist $t_0,C>0$ and $(\aj)\in\R^j$ such that, defining $\eps_j(t) = u(t)-\phib_{\aj}(t)$, one has \[ \forall t\geq t_0,\quad \nh{\eps_j(t)}\leq Ce^{-e_jt}. \] Moreover, defining $\alpha_{j,k}^{\pm}(t) = \int \eps_j(t)Z_k^{\pm}(t)$ for all $k\in\unn$, one has \[\forall k\in\ient{1}{j},\quad \lim_{t\to+\infty} e^{e_kt}\alpha_{j,k}^-(t) = 0. \]
\end{prop}

\begin{rem} \label{initialisation}
As $\eps_1 = u-\phib_{A_1} = \eps+(\phib-\phib_{A_1})$, we have \[ \nh{\eps_1(t)}\leq \nh{\eps(t)}+\nh{\phib(t)-\phib_{A_1}(t)}\leq Ce^{-e_1t}\] by Lemma \ref{th:e1} and \eqref{eq:phiaj}. Moreover, if we define $z_1$ by $z_1(t)=\phib_{A_1}(t)-\phib(t)-A_1e^{-e_1t}Y_1^+(t)$, we have \begin{align*} \alpha_{1,1}^-(t) &= \int \eps_1(t)Z_1^-(t) = \int \eps(t)Z_1^-(t) -A_1e^{-e_1t}\int Y_1^+(t)Z_1^-(t) -\int z_1(t)Z_1^-(t)\\ &= \alpha_1^-(t) -A_1e^{-e_1t}-\int z_1(t)Z_1^-(t) \end{align*} by definition of $\alpha_1^-$ in the previous section and by normalization (iv) of Lemma \ref{th:Zc}. As $\nh{z_1(t)}\leq e^{-(e_1+\c)t}$, we finally deduce, by \eqref{eq:A1}, \[ |e^{e_1t}\alpha_{1,1}^-(t)|\leq |e^{e_1t}\alpha_1^-(t)-A_1| +Ce^{-\c t} \xrightarrow[t\to+\infty]{} 0. \] Therefore, Proposition \ref{th:rec} is proved for $j=1$.
\end{rem}

\begin{proof}[Proof of Proposition \ref{th:rec}]
By remark \ref{initialisation}, it is enough to prove the inductive step: we suppose the assertion true for $j-1$ with $j\geq 2$, and we prove it for $j$. So, suppose that there exist $t_0,C>0$ and $(A_1,\ldots,A_{j-1})\in\R^{j-1}$ such that $\nh{\eps_{j-1}(t)}\leq Ce^{-e_{j-1}t}$ for all $t\geq t_0$, and moreover, for all $k\in\ient{1}{j-1}$, $e^{e_kt} \alpha_{j-1,k}^-(t)\to 0$ as $t\to+\infty$.

\textit{Step 1: Another induction.} Following the proof of Lemma \ref{th:e1}, we prove that if $\nh{\eps_{j-1}(t)}\leq Ce^{-\c_0t}$ with $e_{j-1}\leq \c_0<e_j-\c$, then $\nh{\eps_{j-1}(t)}\leq C'e^{-(\c_0+\c)t}$. But, as $\phib_{A_1}$ is a soliton like $\phib$, estimates \eqref{eq:alpharec} to \eqref{eq:ajrec} of the previous section hold. In other words, we have, with obvious notation, for all $t\geq t_0$, \[ \begin{cases} \forall k\in\unn,\quad \left| \dt \alpha_{j-1,k}^{\pm}(t) \mp e_k\alpha_{j-1,k}^{\pm}(t)\right| \leq Ce^{-\c t}\nh{\eps_{j-1}(t)},\\ \nh{\epst_{j-1}(t)}^2 \leq Ce^{-2\c t}\sup_{t'\geq t} \nh{\eps_{j-1}(t')}^2 +C\|\alphav_{j-1}(t)\|^2,\\ \forall k\in\unn,\quad |a'_{j-1,k}(t)|\leq Ce^{-\c t}\nh{\eps_{j-1}(t)} +C\nh{\epst_{j-1}(t)},\\ \nh{\eps_{j-1}(t)}\leq C\nh{\epst_{j-1}(t)} +C\sum_{k=1}^N |a_{j-1,k}(t)|. \end{cases} \] From these estimates, we deduce the following steps as in the previous section.
\begin{enumerate}[(a)]
\item For all $k\in\unn$, $|\alpha_{j-1,k}^+(t)|\leq Ce^{-(\c_0+\c)t}$.
\item For all $k\in\ient{1}{j-1}$, we have $|{(e^{e_kt}\alpha_{j-1,k}^-(t))}'|\leq Ce^{(e_k-\c_0-\c)t}$. As $e_k-\c_0-\c \leq e_{j-1}-\c_0-\c \leq -\c<0$ and $e^{e_kt} \alpha_{j-1,k}^-(t)\to 0$ as $t\to+\infty$ by hypothesis, we deduce by integration on $[t,+\infty)$ that $|e^{e_kt}\alpha_{j-1,k}^-(t)|\leq Ce^{(e_k-\c_0-\c)t}$, and so $|\alpha_{j-1,k}^-(t)|\leq Ce^{-(\c_0+\c)t}$.
\item For all $k\in\ient{j}{N}$, we still have $|{(e^{e_kt}\alpha_{j-1,k}^-(t))}'|\leq Ce^{(e_k-\c_0-\c)t}$. As $e_k-\c_0-\c \geq e_j-\c_0-\c>0$, we deduce, by integration on $[t_0,t]$, $|e^{e_kt}\alpha_{j-1,k}^-(t)-e^{e_kt_0}\alpha_{j-1,k}^-(t_0)|\leq Ce^{(e_k-\c_0-\c)t}$, and so \[ |\alpha_{j-1,k}^-(t)|\leq Ce^{-e_kt} +Ce^{-(\c_0+\c)t} \leq Ce^{-(\c_0+\c)t}. \]
\item Hence we have $\| \alphav_{j-1}(t)\|^2 \leq Ce^{-2(\c_0+\c)t}$. It follows $\nh{\epst_{j-1}(t)}\leq Ce^{-(\c_0+\c)t}$, $|a_{j-1,k}(t)|\leq Ce^{-(\c_0+\c)t}$ by integration, and finally $\nh{\eps_{j-1}(t)}\leq Ce^{-(\c_0+\c)t}$ as expected.
\end{enumerate}

\textit{Step 2: Identification of $A_j$.} We apply the previous induction until to have $e_j-\c<\c_0<e_j$. Moreover, following the same scheme, we obtain the following estimates.
\begin{enumerate}[(a)]
\item For all $k\in\unn$, $|\alpha_{j-1,k}^+(t)|\leq Ce^{-(\c_0+\c)t}\leq Ce^{-e_jt}$, and we still have \[ |{(e^{e_kt}\alpha_{j-1,k}^-(t))}'|\leq Ce^{(e_k-\c_0-\c)t}. \]
\item For all $k\in\ient{1}{j-1}$, we still have $|\alpha_{j-1,k}^-(t)|\leq Ce^{-(\c_0+\c)t}\leq Ce^{-e_jt}$.
\item For $k=j$, we have $|{(e^{e_jt}\alpha_{j-1,j}^-(t))}'|\leq Ce^{(e_j-\c_0-\c)t} \in L^1([t_0,+\infty))$ since $e_j-\c_0-\c<0$. Thus there exists $A_j\in\R$ such that \[ \lim_{t\to+\infty} e^{e_jt}\alpha_{j-1,j}^-(t)=A_j, \] and moreover $|e^{e_jt}\alpha_{j-1,j}^-(t)-A_j|\leq Ce^{(e_j-\c_0-\c)t}$. Hence we have $|\alpha_{j-1,j}^-(t)|\leq Ce^{-e_jt}$.
\item For all $k\in\ient{j+1}{N}$, we have $e_k-\c_0-\c > e_{j+1}-e_j-\c>0$, thus by integration on $[t_0,t]$, we get $|\alpha_{j-1,k}^-(t)| \leq Ce^{-e_kt}+Ce^{-(\c_0+\c)t}\leq Ce^{-e_jt}$.
\item We now have $\|\alphav_{j-1}(t)\|^2\leq Ce^{-2e_jt}$, and so as in the first step, we conclude that $\nh{\eps_{j-1}(t)}\leq Ce^{-e_jt}$.
\end{enumerate}

\textit{Step 3: Conclusion.} To conclude the induction, we write \begin{align*} \eps_j(t) &= u(t)-\phib_{\aj}(t) = \eps_{j-1}(t) +[\phib_{A_1,\ldots,A_{j-1}}(t) -\phib_{\aj}(t)]\\ &= \eps_{j-1}(t)-A_je^{-e_jt}Y_j^+(t)-z_j(t), \end{align*} where $z_j$, defined by $z_j(t)=\phib_{\aj}(t)-\phib_{A_1,\ldots,A_{j-1}}(t)-A_je^{-e_jt}Y_j^+(t)$, satisfies $\nh{z_j(t)}\leq e^{-(e_j+\c)t}$ by \eqref{eq:phiaj}. Thus, we first have \[ \nh{\eps_j(t)}\leq \nh{\eps_{j-1}(t)} +Ce^{-e_jt}+\nh{z_j(t)}\leq Ce^{-e_jt}. \] Moreover, we find \[ \alpha_{j,k}^-(t) = \int \eps_j(t)Z_k^-(t) = \alpha_{j-1,k}^-(t) -A_je^{-e_jt}\int Y_j^+(t)Z_k^-(t) -\int z_j(t)Z_k^-(t). \] Therefore, for all $k\in\ient{1}{j-1}$, we have $|\alpha_{j,k}^-(t)|\leq |\alpha_{j-1,k}^-(t)|+Ce^{-e_jt} +Ce^{-(e_j+\c)t}$, and so \[ e^{e_kt}|\alpha_{j,k}^-(t)|\leq e^{e_kt}|\alpha_{j-1,k}^-(t)| +Ce^{-(e_j-e_k)t} \xrightarrow[t\to+\infty]{} 0. \] Finally, for $k=j$, we have by (iv) of Lemma \ref{th:Zc}, $\alpha_{j,j}^-(t) = \alpha_{j-1,j}^-(t) -A_je^{-e_jt}-\int z_j(t)Z_j^-(t)$, and so \[ e^{e_jt}|\alpha_{j,j}^-(t)|\leq |e^{e_jt}\alpha_{j-1,j}^-(t)-A_j| +Ce^{-\c t} \xrightarrow[t\to+\infty]{} 0, \] which achieves the proof of Proposition \ref{th:rec}.
\end{proof}

\begin{cor} \label{th:conclusion}
There exist $(\an)\in\R^N$ and $C,t_0>0$ such that, defining $z(t)=u(t)-\phib_{\an}(t)$, we have $\nh{z(t)}\leq Ce^{-2e_Nt}$ for all $t\geq t_0$.
\end{cor}

\begin{proof}
Applying Proposition \ref{th:rec} with $j=N$, we obtain $(\an)\in\R^N$ and $C,t_0>0$ such that $\nh{z(t)}\leq Ce^{-e_Nt}$ for all $t\geq t_0$. Moreover, if we set \[ \alpha_k^{\pm}(t) = \int z(t)Z_k^{\pm}(t) \] for all $k\in\unn$, we have $e^{e_kt}\alpha_k^-(t)\to 0$ as $t\to+\infty$. But, as in the previous proof, it easily follows that if $\nh{z(t)}\leq Ce^{-\c_0t}$ with $\c_0\geq e_N$, then $\nh{z(t)}\leq C'e^{-(\c_0+\c)t}$, and we apply this induction until to have $\c_0=2e_N$.
\end{proof}

\subsection{Uniqueness}

Finally, we prove the following proposition, which achieves the proof of Theorem \ref{th:main}. Note that its proof is based on the schemes developed above, and on arguments developed in \cite[Section 4]{martel:Nsoliton}.

\begin{prop}
There exists $t_0>0$ such that, for all $t\geq t_0$, $z(t)=0$.
\end{prop}

\begin{proof}
We start from the conclusion of Corollary \ref{th:conclusion}, we set \[ \theta(t)=\sup_{t'\geq t} e^{e_Nt'}\nh{z(t')}, \] well defined and decreasing, and we prove that $\theta=0$. Indeed, with obvious notation, we still have the following estimates, for all $t\geq t_0$, \[ \begin{cases} \forall k\in\unn,\quad \left| \dt \alpha_k^{\pm}(t) \mp e_k\alpha_k^{\pm}(t)\right| \leq Ce^{-\c t}\nh{z(t)},\\ \forall k\in\unn,\quad |a'_k(t)|\leq Ce^{-\c t}\nh{z(t)} +C\nh{\zt(t)},\\ \nh{z(t)}\leq C\nh{\zt(t)} +C\sum_{k=1}^N |a_k(t)|. \end{cases} \] Moreover, if we define $H_0$ as in \cite{martel:Nsoliton} by \[ H_0(t) = \int \left\{ \Big( z_x^2(t,x)-F_0(t,z(t,x))\Big)h(t,x)+z^2(t,x)\right\} dx, \] where \[ F_0(t,z) = 2\left[ \frac{{(\phib_{\an}(t)+z)}^{p+1}}{p+1}-\frac{\phib_{\an}^{p+1}(t)}{p+1} -\phib_{\an}^p(t)z \right] \] and $h$ is defined in Section \ref{sec:monotonicity}, we also have $\frac{dH_0}{dt}(t)\geq -Ce^{-2\c t}\nh{z(t)}^2$. Now, we want to prove that $\theta(t)=0$, for $t\geq t_0$ with $t_0$ large enough. Let $t\geq t_0$.

First, we have for all $k\in\unn$, $\left| \dt \alpha_k^{\pm}(t) \mp e_k\alpha_k^{\pm}(t)\right| \leq Ce^{-\c t}e^{-e_Nt}\theta(t)$, and thus, for all $s\geq t$, \[ \left| \dt \alpha_k^{\pm}(s) \mp e_k\alpha_k^{\pm}(s)\right| \leq Ce^{-(e_N+\c)s}\theta(t). \] Hence, we have $|{(e^{-e_ks}\alpha_k^+(s))}'|\leq Ce^{-(e_N+e_k+\c)s} \theta(t)$, and so by integration on $[t,+\infty)$, \[ |\alpha_k^+(t)|\leq Ce^{-(e_N+\c)t}\theta(t). \] Similarly, we have $|{(e^{e_ks}\alpha_k^-(s))}'|\leq Ce^{-(e_N-e_k+\c)s} \theta(t)$, and since $e_N-e_k+\c\geq \c>0$ and $e^{e_kt}\alpha_k^-(t)\to 0$ as $t\to +\infty$, we also get by integration on $[t,+\infty)$, \[ |\alpha_k^-(t)|\leq Ce^{-(e_N+\c)t}\theta(t). \] We thus have $\|\alphav(t)\|^2\leq Ce^{-2(e_N+\c)t}\theta^2(t)$. But we also have, for $s\geq t$, \[ \frac{dH_0}{dt}(s) \geq -Ce^{-2\c s}\nh{z(s)}^2 = -Ce^{-2(e_N+\c)s}{(e^{e_Ns}\nh{z(s)})}^2 \geq -Ce^{-2(e_N+\c)s}\theta^2(t), \] and so by integration on $[t,+\infty)$, $H_0(t)\leq Ce^{-2(e_N+\c)t}\theta^2(t)$. As in the proof of Lemma \ref{th:e1}, we deduce that \[ \nh{\zt(t)}^2 \leq Ce^{-2(e_N+\c)t}\theta^2(t) +C\|\alphav(t)\|^2 \leq Ce^{-2(e_N+\c)t}\theta^2(t), \] and so $\nh{\zt(t)}\leq Ce^{-(e_N+\c)t}\theta(t)$. But, for all $k\in\unn$ and all $s\geq t$, we have \[ |a'_k(s)|\leq Ce^{-\c s}\nh{z(s)}+C\nh{\zt(s)} \leq Ce^{-(e_N+\c)s}\theta(s) \leq Ce^{-(e_N+\c)s}\theta(t), \] and so by integration on $[t,+\infty)$, $|a_k(t)|\leq Ce^{-(e_N+\c)t}\theta(t)$.

Finally, we have shown that there exists $C^*>0$ such that, for all $t\geq t_0$, $\nh{z(t)}\leq C^* e^{-(e_N+\c)t}\theta(t)$. Now fix $t\geq t_0$. We have, for all $t'\geq t$, \[ e^{e_Nt'}\nh{z(t')}\leq C^* e^{-\c t'}\theta(t')\leq C^*e^{-\c t_0}\theta(t), \] and thus $\theta(t)\leq C^*e^{-\c t_0}\theta(t)$. Choosing $t_0$ large enough so that $C^* e^{-\c t_0}\leq \frac{1}{2}$, we obtain $\theta(t)\leq \frac{1}{2}\theta(t)$, so $\theta(t)\leq 0$, and so finally $\theta(t)=0$, as expected.
\end{proof}

\appendix

\section{Appendix} \label{appendix}

\begin{proof}[Proof of Lemma \ref{th:wcf}]
The scheme of the proof is quite similar to the proof of \cite[Theorem 5]{martel:bo}, and uses moreover some arguments developed in \cite[section 3.4]{martel:Nsoliton}. Let $T^* = T^*(\nht{z_0})>0$ be the maximum time of existence of the solution $z(t)$ associated to $z_0$. We distinguish two cases, whether $T<T^*$ or not, and we show that this last case is in fact impossible.

\textbf{First case.} Suppose that $T<T^*$, and let us show that $z_n(T)\cvf z(T)$ in $H^1$. Since $C_0^{\infty}$ is dense in $H^{-1}$ and $\nh{z_n(T)-z(T)}\leq \nh{z_n(T)}+\nh{z(T)} \leq K'$, it is enough to show that $z_n(T)\to z(T)$ in $\mathcal{D}'(\R)$ as $n\to+\infty$. So let $g\in C_0^{\infty}(\R)$ and $\eps>0$, and let us show the lemma in three steps, using a $H^3$ regularization.

\textit{Step 1.} For $N\gg 1$ to fix later, we define $z_{0,n}^N$ and $z_0^N$ by \[ \left\{ \begin{aligned} \widehat{z_{0,n}^N}(\xi) &= \mathbbm{1}_{[-N,N]}(\xi)\widehat{z_{0,n}}(\xi), \\ \widehat{z_0^N}(\xi) &= \mathbbm{1}_{[-N,N]}(\xi)\widehat{z_0}(\xi). \end{aligned} \right. \] In particular, $z_{0,n}^N$ and $z_0^N$ belong to $H^3$, and $z_{0,n}^N\to z_0^N$ in $\mathcal{D}'(\R)$ as $n\to+\infty$, since Fourier transform is continuous in $\mathcal{D}'(\R)$.  Moreover, since $(z_{0,n})$ is uniformly bounded in $H^1$ by Banach-Steinhaus' theorem, we have ${\|z_{0,n}^N\|}_{H^3} \leq C(N)\nh{z_{0,n}}\leq C(N)$, and \begin{align*} \nht{z_{0,n}^N-z_{0,n}}^2 &= \int_{|\xi|\geq N} \Puiss{1+\xi^2}{3/4} \carre{|\widehat{z_{0,n}}(\xi)|}\,d\xi \leq 2^{3/4} \int_{|\xi|\geq N} \puiss{|\xi|}{3/2} \cdot \carre{|\widehat{z_{0,n}}(\xi)|}\,d\xi\\ &\leq \frac{2^{3/4}}{\sqrt N} \int_{|\xi|\geq N} \xi^2 \carre{|\widehat{z_{0,n}}(\xi)|}\,d\xi \leq \frac{2^{3/4}}{\sqrt N}\nh{z_{0,n}}^2 \leq \frac{C}{\sqrt N}, \end{align*} so $z_{0,n}^N \to z_{0,n}$ as $N\to+\infty$ in $\htq$ uniformly in $n$. If we call $z_n^N(t)$ the solution corresponding to initial data $z_{0,n}^N$, and since $\nht{z_n(t)}\leq \nh{z_n(t)}\leq K$, we deduce that \[ \supt \nht{z_n^N(t)-z_n(t)} \leq C\nht{z_{0,n}^N-z_{0,n}} \] for $N$ large enough, by applying \cite[Corollary 2.18]{kpv} with $s=\frac{3}{4}>\frac{p-5}{2(p-1)}$ and $T=T_K = T(\nht{z_n(t)})$. As a consequence, we have \begin{align*} \supt \nht{z_n^N(t)} &\leq \supt \nht{z_n(t)}+C\nht{z_{0,n}^N} +C\nht{z_{0,n}}\\ &\leq \supt \nh{z_n(t)}+2C\nh{z_{0,n}} \leq C. \end{align*} Similarly, since $\supt \nh{z(t)} \leq K'$ by hypothesis, we also obtain, for $N$ large enough, \[ \supt \nht{z^N(t)-z(t)} \leq C'\nht{z_0^N-z_0}, \] where $z^N(t)$ is the solution corresponding to initial data $z_0^N$. Notice that $C$ and $C'$ are independent of $n$, and that by propagation of the regularity, we have $z_n^N(t),z^N(t)\in H^3$ for all $t\in [0,T]$. Finally, we have by the Cauchy-Schwarz inequality \begin{multline*} \left| \int (z_n(T)-z(T))g - \int (z_n^N(T)-z^N(T))g \right| \leq \left| \int (z_n(T)-z_n^N(T))g\right| + \left| \int (z(T)-z^N(T))g\right|\\ \leq (\nld{z_n(T)-z_n^N(T)} + \nld{z(T)-z^N(T)})\nld{g} \leq \frac{C}{\sqrt[4] N} \leq \frac{\eps}{2} \end{multline*} for $N$ large enough, and we now fix it to this value.

\textit{Step 2.} Now that $N$ is fixed, we forget it and the situation amounts in: $z_n(t),z(t)\in H^3$ for all $t\in [0,T]$, $\supt \nht{z_n(t)}\leq C$, ${\|z_{0,n}\|}_{H^3} \leq C'$ (with $C$ and $C'$ independent of $n$) and $z_{0,n}\to z_0$ in $\mathcal{D}'(\R)$ as $n\to+\infty$. The aim of this step is to show consecutively that $z_n(t)$ is uniformly bounded in $H^1$, $H^2$ and $H^3$, and finally $z_n$ is uniformly bounded in $H^1([0,T]\times\R)$.

Since $\supt \nht{z_n(t)}\leq C$ and $\htq(\R) \hookrightarrow L^{\infty}(\R)$ continuously, then we have \[ \supt \nli{z_n(t)}\leq C\quad \m{and}\quad \supt \nld{z_n(t)}\leq C. \] But energy conservation gives, for all $t\in [0,T]$, \[ \frac{1}{2}\int \Carre{\partial_xz_n(t)} -\frac{1}{p+1}\int \puiss{z_n(t)}{p+1} = \frac{1}{2}\int \Carre{\partial_xz_{0,n}} -\frac{1}{p+1}\int z_{0,n}^{p+1}. \] We deduce that: \[ \left|\int \Carre{\partial_x z_n(t)}\right| \leq C\nli{z_n(t)}^{p-1}\nld{z_n(t)}^2 + C\nh{z_{0,n}}^2 +C\nh{z_{0,n}}^{p+1} \leq C, \] and so $\supt \nh{z_n(t)}\leq C$.

To estimate ${\|z_n(t)\|}_{H^2}$, we use the ``modified energy'' as in \cite[Section 3.4]{martel:Nsoliton} (see also \cite{kato}). If we denote $z_n$ by $z$ for a short moment, and if we define $G_2(t)= \int \left( z_{xx}^2(t) -\frac{5p}{3}z_x^2(t)z^{p-1}(t)\right)$ for $t\in [0,T]$, we have the identity \[ G'_2(t) = \frac{1}{12}p(p-1)(p-2)(p-3)\int z_x^5(t) z^{p-4}(t) +\frac{5}{3}p^2(p-1)\int z_x^3(t) z^{2p-3}(t). \] But Gagliardo-Nirenberg inequalities give, for all $k\geq 2$, \[ \int \puiss{|u_x|}{k} \leq C\Puiss{\int u_x^2}{\frac{k+2}{4}} \Puiss{\int u_{xx}^2}{\frac{k-2}{4}}, \] and since $\supt \nli{z(t)}\leq C$, we have \begin{align*} G'_2(t) &\leq C\nli{z(t)}^{p-4}\int \puiss{|z_x(t)|}{5} +C'\nli{z(t)}^{2p-3}\int \puiss{|z_x(t)|}{3}\\ &\leq C\Puiss{\int z_x^2(t)}{7/4} \Puiss{\int z_{xx}^2(t)}{3/4} +C'\Puiss{\int z_x^2(t)}{5/4} \Puiss{\int z_{xx}^2(t)}{1/4}\\ &\leq C\Puiss{\int z_{xx}^2(t)}{3/4} +C'\Puiss{\int z_{xx}^2(t)}{1/4}. \end{align*} Since $a\leq a^{4/3}+1$ and $a\leq a^4+1$ for $a\geq 0$, we deduce that for some $C,D>0$ (still independent of $n$), we have, for all $s\in [0,T]$, \[ G'_2(s)\leq C\left(\int z_{xx}^2(s) \right)+D. \] Now, for $t\in [0,T]$, we integrate between $0$ and $t$, and we obtain \[ G_2(t)-G_2(0)\leq C\int_0^t \nld{z_{xx}(s)}^2\,ds +Dt. \] Moreover, by definition of $G_2$, \begin{align*} \nld{z_{xx}(t)}^2 &\leq \frac{5p}{3}\left| \int z_x^2(t)z^{p-1}(t)\right| +\frac{5p}{3}\left| \int z_x^2(0)z^{p-1}(0)\right|\\ &\qquad + \nld{z_{xx}(0)}^2 +C\int_0^t \nld{z_{xx}(s)}^2\,ds +DT\\ &\leq C\nh{z(t)}^{p+1}+C\nh{z(0)}^{p+1}+{\|z(0)\|}_{H^2} +DT+C\int_0^t \nld{z_{xx}(s)}^2\,ds\\ &\leq B+C\int_0^t \nld{z_{xx}(s)}^2\,ds. \end{align*} Finally, we obtain by Grönwall's lemma that, for all $t\in [0,T]$, \[ \nld{z_{xx}(t)}^2 \leq Be^{Ct} \leq Be^{CT}. \] We can conclude that $\supt {\|z_n(t)\|}_{H^2} \leq C$ with $C>0$ independent of $n$.

For a uniform bound in $H^3$, we use the same arguments as for $H^2$. In fact, it is easier, since we have, by straightforward calculation (we forget again $n$ for a while), \begin{align*} \frac{d}{dt}\int z_{xxx}^2(t) &= -7p(p-1)\int z_{xxx}^2(t)z_x(t)z^{p-2}(t)\\ &\qquad +14p(p-1)(p-2)\int z_{xx}^3(t)z_x(t)z^{p-3}(t)\\ &\qquad +14p(p-1)(p-2)(p-3)\int z_{xx}^2(t)z_x^3(t)z^{p-4}(t)\\ &\qquad +2p(p-1)(p-2)(p-3)(p-4)\int z_{xx}(t)z_x^5(t)z^{p-5}(t). \end{align*} But we have now $\supt \nli{z_x(t)} \leq C\supt \nh{z_x(t)} \leq C\supt {\|z(t)\|}_{H^2} \leq C$, and still $\supt \nli{z(t)}\leq C$, so \[ \frac{d}{dt} \int z_{xxx}^2(t) \leq A\int z_{xxx}^2(t) +B\int \puiss{|z_{xx}(t)|}{3} +C\int z_{xx}^2(t) +D\int |z_{xx}(t)||z_x(t)|. \] Using a Gagliardo-Nirenberg inequality for the second term and the Cauchy-Schwarz one for the last term, we obtain \begin{align*} \frac{d}{dt} \int z_{xxx}^2(t) &\leq A\int z_{xxx}^2(t) +B'\Puiss{\int z_{xx}^2(t)}{5/4}\Puiss{\int z_{xxx}^2(t)}{1/4}\\ &\qquad +C{\|z(t)\|}_{H^2}^2 +D\nld{z_{xx}(t)}\nld{z_x(t)}\\ &\leq A\int z_{xxx}^2(t) +B''\int z_{xxx}^2(t)+B''+C'+D{\|z(t)\|}_{H^2}^2\\ &\leq A'\int z_{xxx}^2(t) +D'. \end{align*} Now, if we integrate this inequality between $0$ and $t\in [0,T]$, we get \begin{align*} \nld{z_{xxx}(t)}^2 &\leq \nld{z_{xxx}(0)}^2 +A'\int_0^t \nld{z_{xxx}(s)}^2\,ds +D't\\ &\leq {\|z(0)\|}_{H^3}^2 +A'\int_0^t \nld{z_{xxx}(s)}^2\,ds +D'T\\ &\leq A'\int_0^t \nld{z_{xxx}(s)}^2\,ds + D'', \end{align*} and we conclude again by Grönwall's lemma that $\nld{z_{xxx}(t)}^2 \leq D''e^{A't} \leq D''e^{A'T}$. Finally, we have the desired bound: $\supt {\|z_n(t)\|}_{H^3} \leq C$.

As $z_{nt}(t) = -z_{nxxx}(t)-pz_{nx}(t)z_n^{p-1}(t)$, then we have, for all $t\in [0,T]$, \[ \nld{z_{nt}(t)} \leq \nld{z_{nxxx}(t)} +p\nli{z_n(t)}^{p-1}\nld{z_{nx}} \leq {\|z_n(t)\|}_{H^3} +C\nh{z_n(t)}^p \leq C. \] We deduce that $(z_n)$ is uniformly bounded in $H^1([0,T]\times\R)$, thus there exists $\tilde z$ such that $z_n \cvf \tilde z$ weakly in $H^1([0,T]\times\R)$ (after passing to a subsequence), and in particular strongly on compacts in $L^2([0,T]\times\R)$. Moreover, since $\sup\nolimits_t {\|z_n(t)\|}_{H^3} \leq C$, we have $\sup\nolimits_t {\|\tilde{z}(t)\|}_{H^3} \leq C$.

\textit{Step 3.} This step is very similar to the first one of the proof of \cite[Theorem 5]{martel:bo}. We recall that we want to prove $\int (z_n(T)-z(T))g\to 0$ as $n\to+\infty$. Let $w_n=z_n-z$. The equation satisfied by $w_n$ is $w_{nt}+w_{nxxx}+ {(z_n^p-z^p)}_x=0$, and moreover \begin{align*} {(z_n^p-z^p)}_x &= pz_{nx}z_n^{p-1}-pz_xz^{p-1} = p[(z_{nx}-z_x)z_n^{p-1}+z_x(z_n^{p-1}-z^{p-1})]\\ &= p\left[w_{nx}z_n^{p-1}+z_x(z_n-z)\sum_{k=0}^{p-2} z_n^kz^{p-2-k} \right]. \end{align*} If we define $S(u,v)=\sum_{k=0}^{p-2} v^k u^{p-2-k}$, the equation satisfied by $w_n$ can be written \[ \begin{cases} w_{nt}+w_{nxxx}+pz_n^{p-1}w_{nx}+pz_x S(z,z_n)w_n=0,\\ w_n(0)=\psi_n = z_{0,n}-z_0. \end{cases} \]

Now consider $v(t)$ the solution of \[ \begin{cases} v_t +v_{xxx} +p{(\tilde{z}^{p-1}v)}_x +pz_xS(z,\tilde{z})v=0,\\ v(T)=g. \end{cases} \] First notice that $\sup\nolimits_t \nld{v}\leq C$ by an energy method. Indeed, we have by direct calculation \[ \frac{d}{dt} \int v^2 = -p\int v^2\left[ (p-1)\tilde{z}_x\tilde{z}^{p-2} +2z_xS(z,\tilde{z})\right]. \] But $\sup\nolimits_t \nli{\tilde{z}_x(t)} \leq \sup\nolimits_t {\|\tilde{z}(t)\|}_{H^2}\leq C$, and similarly $\sup\nolimits_t \nli{\tilde{z}(t)} \leq C$, $\sup\nolimits_t \nli{z_x(t)}\leq C$ and $\sup\nolimits_t \nli{S(z(t),\tilde{z}(t))}\leq C$, and so \[-\frac{d}{ds} \int v^2(s) \leq C\int v^2(s). \] By integration between $t\in [0,T]$ and $T$, we obtain \[\nld{v(t)}^2 -\nld{v(T)}^2 \leq C\int_t^T \nld{v(s)}^2\,ds, \] that is to say $\nld{v(t)}^2 \leq \nld{g}^2 +C\int_t^T \nld{v(s)}^2\,ds$. We conclude, by Grönwall's lemma, that \[ \nld{v(t)}^2 \leq \nld{g}^2 e^{C(T-t)} \leq \nld{g}^2 e^{CT} = K. \]

Now write \[ \int w_n(T,x)g(x)\,dx - \int \psi_n(x)v(0,x)\,dx = \int_0^T\int w_{nt}v +\int_0^T\int w_n v_t = \mathbf{I+II} \] with \[ \left\{ \begin{aligned} \mathbf{I} &= \int_0^T \int w_n\left[ v_{xxx}+p{(vz_n^{p-1})}_x -pz_xS(z,z_n)v\right],\\ \mathbf{II} &= \int_0^T \int w_n\left[ -v_{xxx}-p{(v\tilde{z}^{p-1})}_x +pz_xS(z,\tilde{z})v\right], \end{aligned} \right. \] and so \begin{align*} \mathbf{I+II} &= p\int_0^T \int w_n{[v(z_n^{p-1}-\tilde{z}^{p-1})]}_x + p\int_0^T \int w_nz_xv[S(z,\tilde{z})-S(z,z_n)]\\ &= -p\int_0^T \int w_{nx}v(z_n^{p-1}-\tilde{z}^{p-1}) -p\int_0^T\int w_n z_x v \sum_{k=1}^{p-2} z^{p-2-k}(z_n^k-\tilde{z}^k)\\ &= -p\int_0^T \int w_{nx}v(z_n-\tilde{z})S(\tilde{z},z_n) -p\int_0^T \int w_nz_xv(z_n-\tilde{z})S'(z,\tilde{z},z_n)\\ &= -p\int_0^T \int [w_{nx}S(\tilde{z},z_n)+w_nz_xS'(z,\tilde{z},z_n)]v(z_n-\tilde{z}), \end{align*} where $S(\tilde{z},z_n) = \sum_{k=0}^{p-2} \tilde{z}^{p-2-k}z_n^k$ and $S'(z,\tilde{z},z_n) = \sum_{k=1}^{p-2} \sum_{l=0}^{k-1} z^{p-2-k}\tilde{z}^{k-1-l}z_n^l$ both satisfy \[ \supt \nli{S(\tilde{z},z_n)} \leq C \quad\m{and}\quad \supt \nli{S'(z,\tilde{z},z_n)}\leq C. \] Since $\psi_n \cvf 0$ in $L^2$ and $v(0)\in L^2$, then, for $n$ large enough, $\left| \int \psi_n(x)v(0,x)\,dx\right|\leq \frac{\eps}{4}$. Therefore, it is enough to conclude to show that, for $n$ large enough, $|\mathbf{I}+\mathbf{II}|\leq \frac{\eps}{4}$. But \[ \sup\nolimits_t \nli{w_{nx}S(\tilde{z},z_n)+w_nz_xS'(z,\tilde{z},z_n)}\leq C, \] and $\sup\nolimits_t \nld{z_n-\tilde{z}} \leq C {\|z_n-\tilde{z}\|}_{H^1(]0,T[\times\R)} \leq C$, $\sup\nolimits_t \nld{v}\leq C$. Hence, there exists $R>0$ such that \[ \left| -p\int_0^T \int_{|x|>R} [w_{nx}S(\tilde{z},z_n)+w_nz_xS'(z,\tilde{z},z_n)] v(z_n-\tilde{z}) \right| \leq \frac{\eps}{8}. \] And finally, by Cauchy-Schwarz inequality, we have \begin{multline*} \left| -p\int_0^T \int_{|x|\leq R} [w_{nx}S(\tilde{z},z_n)+w_nz_xS'(z,\tilde{z},z_n)] v(z_n-\tilde{z}) \right| \leq C\int_0^T \int_{|x|\leq R} |z_n-\tilde{z}||v|\\ \leq C\Puiss{\int_0^T \int_{|x|\leq R} \carre{|z_n-\tilde{z}|}}{1/2} \Puiss{\int_0^T\int_{|x|\leq R} v^2}{1/2} \leq C\Puiss{\int_0^T \int_{|x|\leq R} \carre{|z_n-\tilde{z}|}}{1/2} \leq \frac{\eps}{8} \end{multline*} for $n$ large enough, which concludes the first case.

\textbf{Second case.} Suppose that $T^*\leq T$ and let us show that it implies a contradiction. Indeed, there exists $T'<T^*$ such that $\nht{z(T')} \geq 2K$ (where $K$ is the same constant as in the hypothesis of the lemma). But we can apply the first case with $T$ replaced by $T'$, so that $z_n(T')\cvf z(T')$ in $H^1$, and since $\nh{z_n(T')}\leq K$, we obtain by weak convergence $\nht{z(T')}\leq \nh{z(T')} \leq K$, and so the desired contradiction and the end of the proof of the lemma.
\end{proof}

\end{document}